\documentclass[fleqn,preprint,11pt]{elsarticle}

\makeatletter
\def\ps@pprintTitle{
 \let\@oddhead\@empty
 \let\@evenhead\@empty
 \def\@oddfoot{\centerline{\thepage}}%
 \let\@evenfoot\@oddfoot}
\makeatother

\usepackage[margin=3.5cm]{geometry} 

\usepackage{amsmath,amsthm,amsfonts,amssymb,mathrsfs,mathtools}
\usepackage{graphicx}
\usepackage{subfigure}
\usepackage{tabularx}

\usepackage[latin1]{inputenc}
\usepackage{caption} 
\usepackage[shortlabels]{enumitem}

\usepackage{mcode}
\usepackage{comment}

\usepackage{array,booktabs}
\usepackage{siunitx}
\graphicspath{{./figures/}}

\usepackage{algpseudocode}

\usepackage{algorithm}
\usepackage{algorithmicx}

\RequirePackage{xr-hyper}[6.00]
\RequirePackage{ifpdf}[2.3]
\RequirePackage{hyperref}[6.83]
\RequirePackage{xcolor}[2.11]
\colorlet{inlinkcolor}{green!50!black}
\colorlet{exlinkcolor}{red!50!black}
\hypersetup{
  colorlinks = true,
  allcolors = inlinkcolor,
  urlcolor = exlinkcolor,
}

\newtheorem{lemma}{Lemma}

\theoremstyle{definition}

\newcommand{\be}{\begin{equation}}
\newcommand{\ee}{\end{equation}}

\newcommand{\ba}{\begin{aligned}}
\newcommand{\ea}{\end{aligned}}

\newcommand{\bea}{\begin{eqnarray}}
\newcommand{\eea}{\end{eqnarray}}

\newcommand{\ngl}{n_{\rm gl}}
\newcommand{\nsub}{n_{\rm sub}}
\newcommand{\bR}{{\bf R}}
\newcommand{\bP}{{\bf P}}
\newcommand{\bK}{{\bf K}}

\newcommand{\bA}{{\bf A}}
\newcommand{\bM}{{\bf M}}
\newcommand{\bD}{{\bf D}}
\newcommand{\bU}{{\bf U}}
\newcommand{\bV}{{\bf V}}

\def\lund{Centre for Mathematical Sciences, Lund University\\
  Box 118, 221 00 Lund, Sweden}

\def\njit{Department of Mathematical Sciences,
  New Jersey Institute of Technology\\
  Newark, NJ 07102, USA}

\def\ccm{Center for Computational Mathematics, Flatiron Institute, Simons Foundation\\
  New York, NY 10010, USA}

\begin{document}

\begin{frontmatter}

\title{Solving Fredholm second-kind integral equations with singular
  right-hand sides on non-smooth boundaries}

\author[lund]{Johan Helsing}
\address[lund]{\lund}
\ead{johan.helsing@math.lth.se}

\author[njit,ccm]{Shidong Jiang}
\address[njit]{\njit}
\address[ccm]{\ccm}
\ead{shidong.jiang@njit.edu}

\begin{abstract}
   A numerical scheme is presented for the solution of Fredholm
   second-kind boundary integral equations with right-hand sides that
   are singular at a finite set of boundary points. The boundaries
   themselves may be non-smooth. The scheme, which builds on
   recursively compressed inverse preconditioning (RCIP), is universal
   as it is independent of the nature of the singularities. Strong
   right-hand-side singularities, such as $1/|r|^\alpha$ with $\alpha$
   close to $1$, can be treated in full machine precision. Adaptive
   refinement is used only in the recursive construction of the
   preconditioner, leading to an optimal number of discretization
   points and superior stability in the solve phase.
   The performance of
   the scheme is illustrated via several numerical examples, including
   an application to an integral equation derived from the linearized
   BGKW kinetic equation for the steady Couette flow.
\end{abstract}


\end{frontmatter}

\section{Introduction}
\label{sec:intro}

Fredholm second-kind integral equations (SKIEs) have become standard
tools for solving boundary value problems of elliptic partial differential
equations~\cite{colton2012,hsiao2008,kress2014,martinsson2020,pozrikidis1992}.
Advantages include dimensionality reduction
in the solve phase, elimination of the need to impose artificial
boundary conditions for exterior problems, easily achieved
high-order discretization, and optimal complexity when coupled with
fast algorithms such as the fast multipole
method~\cite{greengard1987jcp}.

The present work is about the numerical solution to SKIEs of the form
\begin{equation}
  (I+K)\rho(r)=f(r)\,,\quad r\in\Gamma\,.
  \label{eq:start}
\end{equation}
 Here $r\in\mathbb{R}^2$ is a point in the plane; $I$ is the identity
operator; $\rho$ is an unknown layer density to be solved for;
$\Gamma$ is a piecewise smooth closed contour (boundary) with a finite
number of corners; $K$ is an integral operator on $\Gamma$ which is
compact away from the corners; and $f$ is a right-hand side which is
singular at a finite number of boundary points, which may or may not
coincide with the corner vertices, but otherwise smooth. The union of
corner vertices and boundary points where $f$ is singular is referred
to as {\it singular points}. These points are denoted $\gamma_j$,
$j=1,2,\ldots$. The kernel of $K$ is denoted $K(r,r')$. We also assume
a parameterization $r(s)$ of $\Gamma$ where $s$ is a parameter. 

We shall construct an efficient scheme for the numerical solution
of~(\ref{eq:start}). The difficulty in this undertaking is that the
singularities in $f$ and the non-compactness of $K$ at the $\gamma_j$
may require a very large number of unknowns for the resolution of
$\rho$. This, in turn, may lead to high computing costs and also to
artificial ill-conditioning and reduced achievable precision in
quantities computed from $\rho$.

The nature of the singularity of $f$ can be rather arbitrary in the
applications we consider. There is no analysis of $f$ involved in our
work and neither is there any further analysis of $K$. We only assume
that $f$ can be evaluated everywhere at $\Gamma$ except for at the
$\gamma_j$.  If, however, it is known whether the leading singular
behavior of $f$ is homogeneous on $\Gamma$ in the shape of a wedge and
whether $K$ is scale invariant on such $\Gamma$, that information can
be used to further improve the performance of our scheme.

We shall solve~(\ref{eq:start}) numerically using Nystr{\" o}m
discretization based on underlying composite $16$-point Gauss-Legendre
quadrature and the parameterization $r(s)$ and then accelerate and
stabilize the solution process using an extended version of the
recursively compressed inverse preconditioning (RCIP) method~\cite{Hels18}.
The discretization is chiefly done on a coarse mesh with quadrature panels
of approximately  equal  size. The coarse quadrature panels are
chosen so that the following holds: all singular points $\gamma_j$
coincide with panel endpoints; for $r$ on a panel close to $\gamma_j$
and $r'$ away from $\gamma_j$, $K(r(s),r')$ is smooth; for $r$ away
from $\gamma_j$ and $r'$ on a panel close to $\gamma_j$, $K(r,r'(s))$
is smooth.

The RCIP method assumes that the SKIE has a panelwise smooth
right-hand side and consists of the following steps:
\begin{itemize}
\item Transform the SKIE at hand into a form where the layer density
  to be solved for is panelwise smooth.
\item Use Nystr{\" o}m discretization to discretize the transformed SKIE on
  a grid on a fine mesh, obtained from the coarse mesh by repeated
  subdivision of the panels closest to each $\gamma_j$.
\item Compress the transformed and discretized SKIE so that it can be
  solved on a grid on the coarse mesh without the loss of information.
  This involves the use of a forward recursion.
\item Solve the compressed equation.
\item Reconstruct the solution to the original SKIE from the solution
  to the compressed equation. This involves the use of a backward
  recursion.
\end{itemize}

In this paper, we extend the RCIP method to treat {\it singular}
right-hand sides. The method is {\it universal} since algorithmic steps are
completely independent of the nature of the singularities in the
right-hand-side function. The only information needed is the locations
of singularities. Very strong singularities can be treated in full machine
precision. Indeed, we have studied singularities of the form $1/|r|^\alpha$
for a wide range of $\alpha$ and our method works very well even for
$\alpha=1+0.3{\rm i}$.

We further observe that problems involving sources close to corners
occur surprisingly often in computational electromagnetics and
computational fluid dynamics.  Examples include the determination of
radiation patterns from 5G base stations placed at street
corners~\cite{rappa17} and singularity formation in Hele--Shaw flows
driven by multipoles~\cite{nie2001siap}. These problems involve {\it
nearly singular} right-hand sides and can be treated easily by the
method developed in this paper. The 5G base stations often have
antennas with logarithmic singularities in their radiated field.

Examples of problems involving singular sources close to smooth
surfaces can be found in the area of Internet of Things (IoT), where
the radiation patterns from antennas, integrated in devices, again
need to be found numerically for design purposes~\cite{Devi18}. The
most common types of IoT antennas are stripline and patch antennas.
These have a radiating current density very close to a metallic ground
plane, which often is a smooth part of the housing of the device. As
IoT carrier frequencies get higher, antennas become smaller and the
distance between the current and the housing surface shrinks --
leading to increased need for computational resolution.

The paper is organized as follows.
In Section~\ref{sec:transform}, a new transformation is introduced to treat
the singular right-hand side $f$. Sections~\ref{sec:disc}, \ref{sec:forwrecur},
\ref{sec:hat}, and \ref{sec:backrecur} present detailed modifications to the
other steps in the list that follow from the introduction of the new transform.
Numerical examples are presented in Section~\ref{sec:numerical}. Finally,
we discuss the application of our new method to the integral equation
derived from the linearized Bhatnagar-Gross-Krook-Welander (BGKW) kinetic
equation~\cite{bgk1954,w1954} for the steady Couette flow.

In order to keep the presentation concise, we concentrate on new development and
refer the reader to the recently updated compendium~\cite{Hels18} for
details on the RCIP method. This compendium, in turn, contains
references to original journal papers.

\section{Transformed equation with smooth density}
\label{sec:transform}
This section reviews the original RCIP transformation, discusses its
shortcomings for singular right-hand sides $f$, and presents a new
and better transformation. We frequently use the concept of
{\it panelwise smooth functions}. By this we mean functions which can be
well approximated by polynomials of degree $15$ in $s$ on individual
quadrature panels. We also introduce the boundary subsets
$\Gamma^{j\star}$, which refer to the four panels that are closest to
a point $\gamma_j$ (two on each side), and the boundary subsets
$\Gamma^{j\star\star}$, which refer to the two panels that are
closest to a point $\gamma_j$ (one on each side).

\subsection{The original transformation}
\label{sec:orig}

The original RCIP transformation for SKIEs of the
form~(\ref{eq:start}) assumes that $f$ is a panelwise smooth function
and relies on a kernel split
\begin{equation}
K(r,r')=K^\star(r,r')+K^\circ(r,r')\,,\quad r,r'\in\Gamma\,,
\label{eq:splitKK}
\end{equation}
and a corresponding operator split 
\begin{equation}
K=K^\star+K^\circ\,.
\label{eq:splitK}
\end{equation}
In~(\ref{eq:splitK}), the operator $K^\star$ denotes the part of $K$
that accounts for self-interaction close to the corner vertices $\gamma_j$, and
$K^\circ$ is the compact remainder. The split~(\ref{eq:splitKK}) is
determined from a geometric criterion: if $r$ and $r'$ both are in
$\Gamma^{j\star}$ for some $j$, then $K^\star(r,r')=K(r,r')$.
Otherwise $K^\star(r,r')$ is zero.

The change of variables
\begin{equation}
\rho(r)=(I+K^\star)^{-1}\tilde{\rho}(r)\,,\quad r\in\Gamma\,,
\label{eq:transrho}
\end{equation}
makes~(\ref{eq:start}) with~(\ref{eq:splitK}) assume the form
\begin{equation}
(I+K^\circ(I+K^\star)^{-1})\tilde{\rho}(r)=f(r)\,,\quad r\in\Gamma\,.
\label{eq:prec1}
\end{equation}
 When $f$ is panelwise smooth, the transformed layer density
$\tilde{\rho}$ in~(\ref{eq:prec1}) will, loosely speaking, also be
panelwise smooth. This is so since the action of $K^\circ$ on any
function results in a panelwise smooth function. 
In particular,
$\tilde{\rho}$ will be smooth on $\Gamma^{j\star\star}$. The
panelwise smoothness of $\tilde{\rho}$ on $\Gamma^{j\star\star}$,
inherited from $f$, is the key property which makes the transformed
equation~(\ref{eq:prec1}) efficient for the original
problem~(\ref{eq:start}). The efficiency  comes from the fact that 
a panelwise smooth unknown is easy to resolve by panelwise polynomials.

\subsection{A new transformation}
 When the right-hand side $f$ in~(\ref{eq:start}) is not panelwise smooth,
but has singularities at the corner vertices, the transformed layer
density $\tilde{\rho}$ in~(\ref{eq:prec1}) is not panelwise smooth
either. 
In order to fix this problem we now propose a new transformation of~(\ref{eq:start})
which, in addition to the split of $K$, also splits the right-hand side $f$
and the unknown $\rho$ as
\begin{align}
f(r)&=f^\star(r)+f^\circ(r)\,,\quad r\in\Gamma\,,
\label{eq:splitf}\\
\rho(r)&=v(r)+g(r)\,,\quad r\in\Gamma\,.
\label{eq:rvg}
\end{align}
Here $f^\star(r)=f(r)$ if $r\in\Gamma^{j\star}$ for some $j$.
Otherwise $f^\star(r)$ is zero. The functions $v$ and $g$ are given by
\begin{align}
v(r)&=(I+K^\star)^{-1}\tilde{v}(r)\,,
\label{eq:transv}\\
g(r)&=(I+K^\star)^{-1}f^\star(r)\,,
\label{eq:transg}
\end{align}
where $\tilde{v}$ is a new unknown transformed layer density.

Use of~(\ref{eq:splitK}), (\ref{eq:splitf}), (\ref{eq:rvg}),
(\ref{eq:transv}), and~(\ref{eq:transg}) makes~(\ref{eq:start})
assume the form
\begin{equation}
(I+K^\circ(I+K^\star)^{-1})\tilde{v}(r)=
f^\circ(r)-K^\circ(I+K^\star)^{-1}f^\star(r)\,,\quad r\in\Gamma\,.
\label{eq:prec2}
\end{equation}
One can see, in~(\ref{eq:prec2}), that $\tilde{v}$ is panelwise
smooth on $\Gamma^{j\star\star}$. This is so since the right-hand
side of~(\ref{eq:prec2}) is smooth on $\Gamma^{j\star}$.

We remark that if $\Gamma$ is smooth so that there are no corners,
only singularities in $f$, then the local
transformation~(\ref{eq:transv}) is not needed and~(\ref{eq:prec2})
reduces to
\begin{equation}
  (I+K)v(r)=f^\circ(r)-K^\circ(I+K^\star)^{-1}f^\star(r)\,,\quad
  r\in\Gamma\,.
  \label{eq:prec3}
\end{equation}
One can also imagine mixed situations where~(\ref{eq:transv}) is used
only at those $\gamma_j$ which correspond to corner vertices.

\section{Discretization of~\eqref{eq:prec1} and~\eqref{eq:prec2}}
\label{sec:disc}

\begin{figure}[t]
\centering 
\includegraphics[height=28mm]{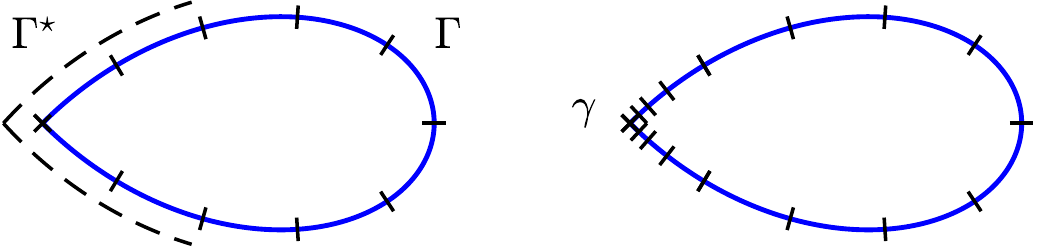}
\caption{\sf A contour $\Gamma$ with a corner at $\gamma$ of opening 
  angle $\theta=\pi/2$. Left: A coarse mesh with ten quadrature panels
  on $\Gamma$. A subset of $\Gamma$, called $\Gamma^\star$, covers the
  four coarse panels closest to $\gamma$. Right: A fine mesh created
  from the coarse mesh by subdividing the panels closest to $\gamma$ a
  number $n_{\rm sub}=3$ of times.}
\label{fig:sno2}
\end{figure}

This section summarizes the modifications needed in the RCIP method in
order for it to apply to~(\ref{eq:prec2}) rather than
to~(\ref{eq:prec1}). We first show how RCIP is applied
to~(\ref{eq:prec1}), with detailed references to~\cite{Hels18}, and
then state the modifications needed for~(\ref{eq:prec2}). For
simplicity of presentation it is from now on assumed that there is
only one singular point, denoted $\gamma$, with an associated
four-panel neighboring zone denoted $\Gamma^\star$ as illustrated in
Figure~\ref{fig:sno2}. We let $\Gamma^{\star\star}$ refer to the two
panels closest to $\gamma$ (one on each side).

The discretization of~(\ref{eq:prec1}) takes place on two different
meshes: on the coarse mesh and on a fine mesh. The fine mesh is
constructed from the coarse mesh by $n_{\rm sub}$ times subdividing
the panels closest to $\gamma$. See the right image of
Figure~\ref{fig:sno2} for an example. Discretization points on the
coarse mesh constitute the {\it coarse grid}. Points on the fine mesh
constitute the {\it fine grid}. The number of refinement levels,
$n_{\rm sub}$, is chosen so that the operator $(I+K^\star)^{-1}$
in~(\ref{eq:prec1}) is resolved to a desired precision.

As mentioned in Section~\ref{sec:intro}, our Nystr{\" o}m discretization
relies on composite 16-point Gauss--Legendre quadrature. This means
that an integral
\begin{displaymath}
\int_\Gamma h(r)\,{\rm d}\ell\,,
\end{displaymath}
where $h$ is a smooth function and ${\rm d}\ell$ is an element of arc
length, can be approximated by a sum on the coarse grid
\begin{equation}
\int_\Gamma h(r)\,{\rm d}\ell
\approx\sum_j h(r(s_{{\rm coa}_j}))
\lvert\dot{r}(s_{{\rm coa}_j})\rvert w_{{\rm coa}_j}\,.
\label{eq:basdisc}
\end{equation}
Here $\dot{r}(s)={\rm d}r(s)/{\rm d}s$ denotes differentiation with respect
to the boundary parameter $s$ and $w_{{\rm coa}_j}$ are appropriately scaled
Gauss--Legendre weights. 
A formula, analogous
to~(\ref{eq:basdisc}) but with subscripts {\footnotesize coa} replaced
with subscripts {\footnotesize fin}, holds when $h$ is a singular
function that needs the fine grid for resolution.

\subsection{Smooth right-hand side}

It is shown in~\cite[Appendix B]{Hels18} that the discretization
of~(\ref{eq:prec1}) on the fine grid followed by RCIP-style
compression leads to the system
\begin{equation}
\left({\bf I}_{\rm coa}+{\bf K}_{\rm coa}^\circ{\bf R}\right)
\tilde{\boldsymbol{\rho}}_{\rm coa}={\bf f}_{\rm coa}\,.
\label{eq:disc1}
\end{equation}
Here ${\bf R}$ is a block-diagonal matrix defined
as~\cite[Eq.~(26)]{Hels18}
\begin{equation}
{\bf R}={\bf P}_W^{\rm T}
\left({\bf I}_{\rm fin}+{\bf K}_{\rm fin}^\star\right)^{-1}
{\bf P}\,,
\label{eq:R}
\end{equation}
and the subscripts {\footnotesize coa} and {\footnotesize fin} indicate what
type of mesh is used for discretization. The prolongation matrix ${\bf P}$
interpolates piecewise polynomial functions known at the coarse grid
to the fine grid. The weighted prolongation matrix ${\bf P}_W$
resembles ${\bf P}$, but is designed to act on discretized functions
multiplied by quadrature weights~\cite[Eq.~(21)]{Hels18}. The
interpolation is done with respect to the boundary parameter $s$. The
superscript {\footnotesize T} denotes the transpose.

\subsection{Singular right-hand side}

The discretization and compression of~(\ref{eq:prec2}) can be
constructed in a manner completely analogous to that
of~(\ref{eq:prec1}) and leads to the system
\begin{equation}
({\bf I}_{\rm coa}+{\bf K}_{\rm coa}^\circ{\bf R})
\tilde{\bf v}_{\rm coa}=
{\bf f}^\circ_{\rm coa}
-{\bf K}_{\rm coa}^\circ {\bf R}_f{\bf f}^\star_{\rm coa}\,,
\label{eq:disc2}
\end{equation}
where ${\bf R}$ is as in~(\ref{eq:R}) and
\begin{equation}
{\bf R}_f=
{\bf P}_W^{\rm T}
\left({\bf I}_{\rm fin}+{\bf K}_{\rm fin}^\star\right)^{-1}
{\bf P}_f\,.
\label{eq:Rf}
\end{equation}
The prolongation matrix ${\bf P}_f$ interpolates $f$ from the coarse
grid to the fine grid and can easily be constructed from computed
values of $f$ on the two grids respectively. The matrix ${\bf P}_f$ is
block-diagonal with blocks being either identity matrices or, for
entries corresponding to discretization points on
$\Gamma^{\star\star}$, the rectangular rank-one matrix
\begin{equation}
\bP_f^{\star\star}= \frac{1}
   {{\bf f}_{\rm coa}^{\star\star{\rm H}}{\bf f}_{\rm coa}^{\star\star}}
   {\bf f}^{\star\star}_{\rm fin}{\bf f}^{\star\star{\rm H}}_{\rm coa}\,.
\end{equation}
Here ${\bf f}_{\rm fin}^{\star\star}$ is a column vector with values
of $f$ on the fine grid on $\Gamma^{\star\star}$,
${\bf f}_{\rm coa}^{\star\star}$ is a column vector with values of $f$ on the
coarse grid on $\Gamma^{\star\star}$, and {\footnotesize H} denotes
the conjugate transpose. Clearly,
$\bP_f^{\star\star}{\bf f}_{\rm coa}^{\star\star}={\bf f}_{\rm fin}^{\star\star}$.

\section{Forward recursion for ${\bf R}$ and ${\bf R}_f$}
\label{sec:forwrecur}

\begin{figure}[t]
\centering 
\includegraphics[height=31mm]{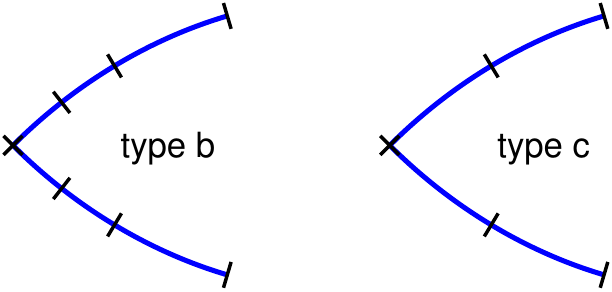}
\includegraphics[height=31mm]{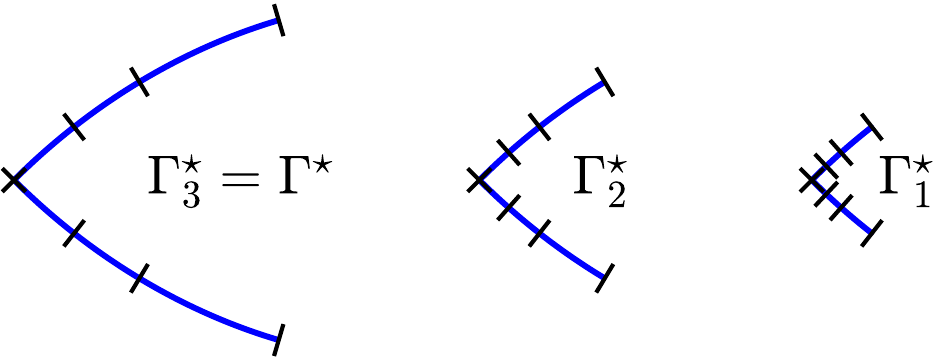}
\caption{\sf 
  Top row: meshes of type {\sf b} and type {\sf c} on the boundary
  subset $\Gamma^\star$. The type {\sf b} mesh has six panels. The type
  {\sf c} mesh has four panels. Bottom row: the boundary subsets
  $\Gamma_3^\star=\Gamma^\star$, $\Gamma_2^\star$, and
  $\Gamma_1^\star$ along with their corresponding type {\sf b} meshes
  for $n_{\rm sub}=3$.}
\label{fig:typasub}
\end{figure}

The matrices ${\bf R}$ and ${\bf R}_f$ of~(\ref{eq:R})
and~(\ref{eq:Rf}) differ from the identity matrix only in that they
each contain a non-trivial $64\times 64$ diagonal block, associated
with entries corresponding to discretization points on $\Gamma^\star$.
These diagonal blocks can be efficiently constructed via recursions
relying on the discretization of $K$ on small meshes on a hierarchy of
boundary subsets $\Gamma_i^\star$ around $\gamma$. Central to the
recursions are two types of overlapping meshes called type {\sf b} and
type {\sf c}. The type {\sf b} mesh contains six quadrature panels.
The type {\sf c} mesh contains four quadrature panels. See
Figure~\ref{fig:typasub} for an illustration
and~\cite[Section~7]{Hels18} for more information.

The recursions for ${\bf R}$ and ${\bf R}_f$, given below, use local
prolongation matrices ${\bf P}_{\rm bc}$, ${\bf P}_{W{\rm bc}}$, and
${\bf P}_{fi{\rm bc}}$ similar to the global matrices ${\bf P}$, ${\bf P}_W$
and ${\bf P}_f$ of Section~\ref{sec:disc}. The matrix ${\bf P}_{\rm bc}$
performs polynomial interpolation from a grid on a type
{\sf c} mesh to a grid on a type {\sf b} mesh on the same level. The
matrix ${\bf P}_{fi{\rm bc}}$ interpolates $f$ from a grid on a type
{\sf c} mesh to a grid on a type {\sf b} mesh on the same level. The
subscript {\footnotesize $i$} indicates that ${\bf P}_{fi{\rm bc}}$
depends on the hierarchical level in which it appears.

The matrices ${\bf P}_{\rm bc}$ and ${\bf P}_{W{\rm bc}}$ are level
independent. Their construction is described
in~\cite[Section~7.1]{Hels18}. The matrix ${\bf P}_{fi{\rm bc}}$ is
constructed analogously to the matrix ${\bf P}_f$ of
Section~\ref{sec:disc}.

\subsection{The recursion for ${\bf R}$}

It is shown in~\cite[Appendix D]{Hels18} that the non-trivial diagonal
$64\times 64$ block of ${\bf R}$ can be obtained from the recursion
\begin{equation}
{\bf R}_i={\bf P}^T_{W\rm{bc}}
\left(
\mathbb{F}\{{\bf R}_{i-1}^{-1}\}+{\bf I}_{\rm b}^\circ+{\bf K}_{i{\rm
    b}}^\circ\right)^{-1}{\bf P}_{\rm{bc}}\,,
\quad i=1,\ldots,n_{\rm sub}\,,
\label{eq:finalR}
\end{equation}
with the initializer
\begin{equation}
\mathbb{F}\{{\bf R}_0^{-1}\}=
{\bf I}_{\rm b}^\star+{\bf K}^\star_{1{\rm b}}\,.
\label{eq:Rstart}
\end{equation}
Here ${\bf K}_{i{\rm b}}$ is the discretization of $K$ on a type {\sf b}
mesh on level $i$ in the hierarchy of local meshes around
$\gamma$. The operator $\mathbb{F}\{\cdot\}$ expands its matrix
argument by zero-padding (adding a frame of zeros of width 16 around
it). The superscripts {\footnotesize $\star$} and {\footnotesize $\circ$}
denote matrix splits analogous to the split~(\ref{eq:splitK}).

The forward recursion~(\ref{eq:finalR}) starts at the finest
refinement level, $i=1$, and ascends through the hierarchy of levels
until it reaches the coarsest level $i=n_{\rm sub}$. The matrix
${\bf R}_{n_{\rm sub}}$ is equal to the non-trivial $64\times 64$ block of
${\bf R}$.

\subsection{The recursion for ${\bf R}_f$}

A recursion for the non-trivial $64\times 64$ block of ${\bf R}_f$ can
be derived in complete analogy with the derivation
of~(\ref{eq:finalR}). See \ref{sec:rfderiv} for the key steps of the
derivation. The result is
\begin{multline}
{\bf R}_{fi}={\bf P}_{W{\rm bc}}^T
\left(\mathbb{F}\{{\bf R}_{i-1}^{-1}\}+{\bf I}_{\rm b}^\circ+
{\bf K}_{i{\rm b}}^\circ\right)^{-1}
\left(\mathbb{F}\{{\bf R}_{i-1}^{-1}{\bf R}_{f(i-1)}\}
+{\bf I}_{\rm b}^\circ\right)
{\bf P}_{fi{\rm bc}}\,,\\
\quad i=1,\ldots,n_{\rm sub}\,.
\label{eq:finalRf}
\end{multline}
The recursion~(\ref{eq:finalRf}) can be initialized with
\begin{equation}
{\bf R}_{f0}={\bf R}_0
\label{eq:Rfstart}
\end{equation}
and run in tandem with~(\ref{eq:finalR}).

\subsection{Further improvement of the recursions}
\label{sec:improvement}
From~(\ref{eq:disc2}) it is evident that the matrix ${\bf R}_f$ acts
only on one particular known vector, namely on ${\bf f}^\star_{\rm coa}$.
Therefore, rather than first finding ${\bf R}_f$
via~(\ref{eq:finalRf}) and then computing the vector
\begin{equation}
{\bf r}_f^\star={\bf R}_f{\bf f}^\star_{\rm coa}\,,
\end{equation}
one can instead modify~(\ref{eq:finalRf}) with~(\ref{eq:Rfstart}) so
that it produces ${\bf r}_f^\star$ directly
\begin{gather}
%\begin{multline}
{\bf r}_{fi}^\star={\bf P}_{W{\rm bc}}^T
\left(\mathbb{F}\{{\bf R}_{i-1}^{-1}\}+{\bf I}_{\rm b}^\circ+
{\bf K}_{i{\rm b}}^\circ\right)^{-1}
\left(\mathbb{F}\{{\bf R}_{i-1}^{-1}{\bf r}_{f(i-1)}^\star\}
%+{\bf f}_{i{\rm b}}^\circ\right)\,,\\
%\quad i=1,\ldots,n_{\rm sub}\,,\label{eq:finalRfvec}
+{\bf f}_{i{\rm b}}^\circ\right)\,,\nonumber\\
\hspace{3.5in} i=1,\ldots,n_{\rm sub}\,,\label{eq:finalRfvec}\\
%\end{multline}
%\begin{equation}
{\bf r}_{f0}^\star={\bf R}_{f0}{\bf f}_{1{\rm b}}^\star\,.
\label{eq:Rfvecstart}
%\end{equation}
\end{gather}
The recursion~(\ref{eq:finalRfvec}) with~(\ref{eq:Rfvecstart}) is a
bit faster than~(\ref{eq:finalRf}) with~(\ref{eq:Rfstart}).

It is also worth noting that all three recursions~(\ref{eq:finalR}),
(\ref{eq:finalRf}), and (\ref{eq:finalRfvec}) can be implemented
without the explicit inversion of ${\bf R}_{i-1}$. The key to this is
to use the Schur--Banachiewicz inverse formula for partitioned
matrices~\cite[Eq.~(8)]{Hend81}. Avoiding inversion is particularly
important when ${\bf R}_{i-1}$ is ill conditioned. Details on an
implementation for~(\ref{eq:finalR}), free from inversion of
${\bf R}_{i-1}$, are given in~\cite[Section~8]{Hels18}
(see also~\ref{sec:Frecursiondetail}).

\subsection{Efficient initializers}

A minor problem with the recursions~(\ref{eq:finalR}),
(\ref{eq:finalRf}), and (\ref{eq:finalRfvec}) is that it could be
difficult to determine a suitable recursion length $n_{\rm sub}$ a
priori. A too large $n_{\rm sub}$ leads to unnecessary work. A too
small $n_{\rm sub}$ fails to resolve the problem. Should it, however,
happen that $K$ is scale-invariant on wedges, there may be an easy way
around this problem. The key observation is that for large $n_{\rm sub}$
and at levels $i$ such that $n_{\rm sub}-i\gg 1$, the matrices
${\bf K}^\circ_{i{\rm b}}$ have often converged to a matrix ${\bf K}^\circ_{\rm b}$
(in double precision arithmetic) that is
independent of $i$. This, typically, happens for $n_{\rm sub}-i>60$
and means that~(\ref{eq:finalR}) assumes the form of a fixed-point
iteration
\begin{equation}
{\bf R}_{\ast i}={\bf P}^T_{W\rm{bc}}
\left(
\mathbb{F}\{{\bf R}_{\ast(i-1)}^{-1}\}+{\bf I}_{\rm b}^\circ+{\bf K}_{\rm
    b}^\circ\right)^{-1}{\bf P}_{\rm{bc}}\,,
\quad i=1,2,\ldots\,.
\label{eq:fixedR}
\end{equation}
It also means that all ${\bf R}_i$ in~(\ref{eq:finalR}) are the same
for $n_{\rm sub}-i\gg 1$. In view of the above one can replace the
initializer ${\bf R}_0$ of~(\ref{eq:Rstart}) with the fixed-point
matrix ${\bf R}_\ast$ obtained by running~(\ref{eq:fixedR}) until
convergence. With the choice ${\bf R}_0={\bf R}_\ast$, it is enough to
take $n_{\rm sub}=60$ steps in the recursion~(\ref{eq:finalR}). See,
further, the discussion in~\cite[Sections~12--13]{Hels18}.

In a procedure similar to that just described, it is also possible to
replace the initializers ${\bf R}_{f0}$ and ${\bf r}_{f0}^\star$
of~(\ref{eq:Rfstart}) and~(\ref{eq:Rfvecstart}) with more efficient
initializers. The requirements are, in addition to that $K$ is scale
invariant on wedges, that the leading singular behavior of $f$ is
homogeneous on wedges. If this holds, and if $n_{\rm sub}-i\gg 1$,
then~(\ref{eq:finalRf}) assumes the form of a linear fixed-point
iteration
\begin{multline}
{\bf R}_{f\ast i}={\bf P}_{W{\rm bc}}^T
\left(\mathbb{F}\{{\bf R}_\ast^{-1}\}+{\bf I}_{\rm b}^\circ+
{\bf K}_{\rm b}^\circ\right)^{-1}
\left(\mathbb{F}\{{\bf R}_\ast^{-1}{\bf R}_{f\ast(i-1)}\}
+{\bf I}_{\rm b}^\circ\right)
{\bf P}_{f{\rm bc}}\,,\\
\quad i=1,2,\ldots\,.
\label{eq:fixedRf}
\end{multline}
The fixed-point matrix ${\bf R}_{f\ast}$ can be found with direct
methods solving a Sylvester equation. With the choices
${\bf R}_{f0}={\bf R}_{f\ast}$ and
${\bf r}_{f0}^\star={\bf R}_{f\ast}{\bf f}_{1{\rm b}}^\star$
it is enough to take $n_{\rm sub}=60$ steps in
the recursions~(\ref{eq:finalRf}) and~(\ref{eq:finalRfvec}).

We remark that efficient initializers can be found also under more
general conditions on $K$ than scale invariant on wedges.
See~\cite[Section~5.3]{HelsJian18}, for an example.

\section{Computing integrals of $\rho$}
\label{sec:hat}

Often, in applications, one is not primarily interested in the
solution $\rho$ to~(\ref{eq:start}) in itself. Rather, one is
interested in computing functionals of $\rho$ of the type
\begin{equation}
q=\int_\Gamma h(r)\rho(r)\,{\rm d}\ell\,,
\label{eq:q0}
\end{equation}
where $h(r)$ is a smooth function. The RCIP method offers an elegant
way to do this that only involves quantities appearing in the
compressed equations~(\ref{eq:disc1}) and~(\ref{eq:disc2}).

Assume that $f$ is smooth and consider~(\ref{eq:disc1}). The quantity
$q$ of~(\ref{eq:q0}) can then be well approximated by the sum on the
coarse grid
\begin{equation}
q\approx\sum_j h(r(s_{{\rm coa}_j}))\hat{\rho}_{{\rm coa}_j}
\lvert\dot{r}(s_{{\rm coa}_j})\rvert w_{{\rm coa}_j}\,,
\label{eq:q1}
\end{equation}
where $\hat{\rho}_{{\rm coa}_j}$ are elements of the {\it
  weight-corrected} density vector
\begin{equation}
\hat{\boldsymbol{\rho}}_{\rm coa}
={\bf R}\tilde{\boldsymbol{\rho}}_{\rm coa}\,.
\label{eq:rhotilde1}
\end{equation}
See~\cite[Appendix C]{Hels18} for a proof.

The situation for a singular $f$ and~(\ref{eq:disc2}) is completely
analogous. The expression~(\ref{eq:q1}) holds
with~(\ref{eq:rhotilde1}) replaced by
\begin{equation}
\hat{\boldsymbol{\rho}}_{\rm coa}
={\bf R}\tilde{\bf v}_{\rm coa}+{\bf r}_f^\star\,.
\label{eq:rhotilde2}
\end{equation}

\section{Backward recursion for the reconstruction of $\rho$}
\label{sec:backrecur}

When the compressed equation~(\ref{eq:disc2}) has been solved for
$\tilde{\bf v}_{\rm coa}$ one might also be interested in the
reconstruction of the discretized solution 
\begin{equation}
\boldsymbol{\rho}_{\rm fin}={\bf v}_{\rm fin}+{\bf g}_{\rm fin}
\label{eq:fff}
\end{equation}
to the original SKIE~(\ref{eq:start}), compare~(\ref{eq:rvg}). Such a
reconstruction can be achieved by, loosely speaking, running the
recursions~(\ref{eq:finalR}) and~(\ref{eq:finalRf}) backward on
$\Gamma^\star$. Outside of $\Gamma^\star$, the coarse grid and the
fine grid coincide and it holds that $\boldsymbol{\rho}={\bf v}=\tilde{\bf v}$,
see~(\ref{eq:rvg}), (\ref{eq:transv}), and (\ref{eq:transg}).

\subsection{The recursion for ${\bf v}_{\rm fin}$}

We first review the reconstruction of $\boldsymbol{\rho}_{\rm fin}$
from the solution $\tilde{\boldsymbol{\rho}}_{\rm coa}$
to~(\ref{eq:disc1}). The mechanism for this reconstruction was
originally derived in~\cite[Section~7]{Hels09JCP} and is also
summarized in~\cite[Section 10]{Hels18}. The backward recursion reads
\be
\vec{\boldsymbol{\rho}}_{{\rm coa},i}=
\left[{\bf I}_{\rm b}-{\bf K}_{i{\rm b}}^\circ
\left(\mathbb{F}\{{\bf R}_{i-1}^{-1}\}+
{\bf I}_{\rm b}^\circ+{\bf K}_{i{\rm b}}^\circ
\right)^{-1}\right]{\bf P}_{\rm bc}\tilde{\boldsymbol{\rho}}_{{\rm coa},i}\,,
\quad i=n_{\rm sub},\ldots,1\,.
\label{eq:backrho}
\ee
Here $\tilde{\boldsymbol{\rho}}_{{\rm coa},i}$ is a column vector with
$64$ elements. In particular, $\tilde{\boldsymbol{\rho}}_{{\rm coa},n_{\rm sub}}$
is the restriction of
$\tilde{\boldsymbol{\rho}}_{\rm coa}$ to $\Gamma^\star$, while
$\tilde{\boldsymbol{\rho}}_{{\rm coa},i}$ are taken as elements
$\{17:80\}$ of $\vec{\boldsymbol{\rho}}_{{\rm coa},i+1}$ for $i<n_{\rm sub}$.
The elements $\{1:16\}$ and $\{81:96\}$ of
$\vec{\boldsymbol{\rho}}_{{\rm coa},i}$ are the reconstructed values
of $\boldsymbol{\rho}_{\rm fin}$ on the outermost panels of a type
{\sf b} mesh on $\Gamma_i^\star$. 

When the recursion is completed, there are no values assigned to
$\boldsymbol{\rho}_{\rm fin}$ at points on the four innermost panels
(on $\Gamma_1^\star$ closest to $\gamma$) on the fine grid.
Reconstructed weight-corrected values of $\boldsymbol{\rho}_{\rm fin}$
on these panels can then be used, rather than true values, and are
obtained from
\begin{equation}
{\bf R}_0\tilde{\boldsymbol{\rho}}_{{\rm coa},0}\,.
\label{eq:recendR}
\end{equation}

We now observe that the reconstruction of ${\bf v}_{\rm fin}$ from the
solution $\tilde{\bf v}_{\rm coa}$ to~(\ref{eq:disc2}) is identical to
the reconstruction just described. This is so since both
$\tilde{\rho}$ of~(\ref{eq:prec1}) and $\tilde{v}$ of~(\ref{eq:prec2})
are panelwise smooth functions.

\subsection{The recursion for ${\bf g}_{\rm fin}$}

The backward recursion for ${\bf g}_{\rm fin}$ from ${\bf f}_{\rm coa}^\star$
is analogous to~(\ref{eq:backrho}), but the vector
$\vec{\bf g}_{{\rm coa},i}$, corresponding to
$\vec{\boldsymbol{\rho}}_{{\rm coa},i}$ in~(\ref{eq:backrho}), needs a
split in a singular and a panelwise smooth part on a type {\sf b} mesh
on each $\Gamma_i^\star$
\begin{equation}
\vec{\bf g}_{{\rm coa},i}
=\vec{\bf g}_{{\rm coa},i}^{\,{\rm smo}}+{\bf f}_{i{\rm b}}\,.
\label{eq:fg}
\end{equation}
The backward recursion can then be written
\begin{multline}
\vec{\bf g}_{{\rm coa},i}^{\,{\rm smo}}=
\left[{\bf I}_{\rm b}-{\bf K}_{i{\rm b}}^\circ
\left(\mathbb{F}\{{\bf R}_{i-1}^{-1}\}+
{\bf I}_{\rm b}^\circ+{\bf K}_{i{\rm b}}^\circ
\right)^{-1}\right]{\bf P}_{\rm bc}\tilde{\bf g}_{{\rm coa},i}^{\,{\rm smo}}\\
-{\bf K}_{i{\rm b}}^\circ
\left(\mathbb{F}\{{\bf R}_{i-1}^{-1}\}+{\bf I}_{\rm b}^\circ+
{\bf K}_{i{\rm b}}^\circ\right)^{-1}
\left(\mathbb{F}\{{\bf R}_{i-1}^{-1}{\bf r}_{f(i-1)}^\star\}
+{\bf f}_{i{\rm b}}^\circ\right)\,,\quad i=n_{\rm sub},\ldots,1\,.
\label{eq:backg2}
\end{multline}
Here $\tilde{{\bf g}}_{{\rm coa},i}^{\,{\rm smo}}$ is a column vector
with $64$ elements. In particular,
$\tilde{{\bf g}}_{{\rm coa},n_{\rm sub}}^{\,{\rm smo}}={\bf 0}$ while
$\tilde{{\bf g}}_{{\rm coa},i}^{\,{\rm smo}}$ are taken as elements $\{17:80\}$ of
$\vec{{\bf g}}_{{\rm coa},i+1}$ for $i<n_{\rm sub}$. The elements
$\{1:16\}$ and $\{81:96\}$ of $\vec{{\bf g}}_{{\rm coa},i}$
in~(\ref{eq:fg}) are the reconstructed values of ${\bf g}_{\rm fin}$
on the outermost panels of a type {\sf b} mesh on $\Gamma_i^\star$.
The reconstructed weight-corrected values of ${\bf g}_{\rm fin}$ on
the four innermost panels on the fine grid are obtained from
\begin{equation}
{\bf R}_0\tilde{{\bf g}}_{{\rm coa},0}^{\,{\rm smo}}+{\bf r}_{f0}^\star\,.
\label{eq:recendT}
\end{equation}

\section{Numerical examples}
\label{sec:numerical}
We now demonstrate the efficiency of our numerical scheme
for~(\ref{eq:start}). The scheme consists of the compressed
equation~(\ref{eq:disc2}), the recursions~(\ref{eq:finalR}),
(\ref{eq:finalRfvec}), (\ref{eq:backrho}), and (\ref{eq:backg2}), and
initializers obtained from~(\ref{eq:fixedR}) and~(\ref{eq:fixedRf}).
The code is implemented in {\sc Matlab}, release 2020b, and executed
on a 64 bit Linux laptop with a 2.10GHz Intel i7-4600U CPU.
The implementations are standard and rely on built-in functions such as
{\tt dlyap} (SLICOT subroutine SB04QD), for the Sylvester equation.
Large linear systems are solved using GMRES, incorporating a
 low-threshold stagnation avoiding technique~\cite[Section~8]{Hels08} 
applicable to systems coming from discretizations of SKIEs.
The GMRES stopping criterion is set to machine
epsilon in the estimated relative residual.

\subsection{A transmission problem for Laplace's equation}
\subsubsection{Test equation and test geometry}

We solve the SKIE~(\ref{eq:start}) with an integral operator $K$
defined by its action on $\rho$ as
\begin{equation}
K\rho(r)=2\lambda\int_{\Gamma}\frac{\partial G}{\partial\nu}(r,r')
\rho(r')\,{\rm d}\ell'\,,\quad r\in \Gamma\,.
\label{eq:Kdef}
\end{equation}
Here $\lambda$ is a parameter set to $\lambda=0.5$, $\nu(r)$ is the
exterior unit normal at position $r$ on $\Gamma$,
$\partial/\partial\nu=\nu(r)\cdot\nabla$, $G(r,r')$ is the fundamental
solution to Laplace's equation in the plane
\begin{equation}
G(r,r')=-\frac{1}{2\pi}\log|r-r'|\,,
\end{equation}
and $\Gamma$ is the closed contour with a corner at $\gamma=0$
parameterized as
\begin{equation}
r(s)=\sin(\pi s)\left(\cos((s-0.5)\theta),\sin((s-0.5)\theta)\right)\,,
\quad s\in[0,1]\,,
\label{eq:gamma}
\end{equation}
where $\theta$ corresponds to the opening angle of the corner. With
the particular choice~(\ref{eq:Kdef}) for $K$, the
SKIE~(\ref{eq:start}) can model a transmission problem for Laplace's
equation~\cite[Section~4]{Hels18}.

In our experiments we shall use a coarse mesh that is sufficiently
refined as to resolve~(\ref{eq:start}) away from $\gamma$, vary the
number of distinct recursion steps $n_{\rm sub}$, and monitor the
convergence of the scalar quantity $q$ of~(\ref{eq:q0}) with $h(r)=1$.
For comparison we compute $q$ in two ways: first on the coarse grid
via~(\ref{eq:q1}) and with $\hat{\boldsymbol{\rho}}_{\rm coa}$
from~(\ref{eq:rhotilde2}), then on the fine grid via
\begin{equation}
q\approx\sum_j \rho_{{\rm fin}_j}\lvert\dot{r}(s_{{\rm fin}_j})\rvert 
w_{{\rm fin}_j}
\label{eq:q2}
\end{equation}
and with $\boldsymbol{\rho}_{\rm fin}$ from~(\ref{eq:fff}).

\subsubsection{Example with an analytical solution}
\label{sec:ana}

We start with an example where the corner opening angle is set to
$\theta=\pi$. Then $\Gamma$ of~(\ref{eq:gamma})  becomes  a
circle with a circumference of $\pi$. The right-hand side
of~(\ref{eq:start}) is set to
\begin{equation}
f(r)=\frac{1}{\ell^\alpha(r)}+\frac{1}{(\pi-\ell(r))^\alpha}\,,
\label{eq:fcirc}
\quad r\in\Gamma\,.
\end{equation}
Here $\alpha$ is a possibly complex parameter with $\Re{\rm e}\{\alpha\}<1$,
controlling the strength of the singularity at
$\gamma$, and $\ell(r)$ is the distance from $\gamma$ to $r$ measured
in arc length along $\Gamma$ as the circle is traversed in a
counterclockwise fashion. The solution $\rho$ to~(\ref{eq:start}) can
in this example be computed analytically
\begin{equation}
\rho(r)=f(r)+\frac{2\lambda\pi^{-\alpha}}{(1-\alpha)(1-\lambda)}\,,
\quad r\in\Gamma\,,
\label{eq:rhoana}
\end{equation}
as can the quantity $q$:
\begin{equation}
q=\frac{2\pi^{(1-\alpha)}}{(1-\alpha)(1-\lambda)}\,.
\label{eq:qana}
\end{equation}
Note that $\rho(r)$ of~(\ref{eq:rhoana}) diverges everywhere on
$\Gamma$ as  $\alpha\to 1^-$, so there is no solution for
$\alpha=1$. Neither is there a finite limit value of $q$.
If $\Im{\rm m}\{\alpha\}\ne 0$ is fixed, however, then there is a limit
solution $\rho(r)$ and a finite limit value of $q$ as
$\Re{\rm e}\{\alpha\}\to 1^-$. 

\begin{figure}[t]
\centering
   \includegraphics[height=42mm]{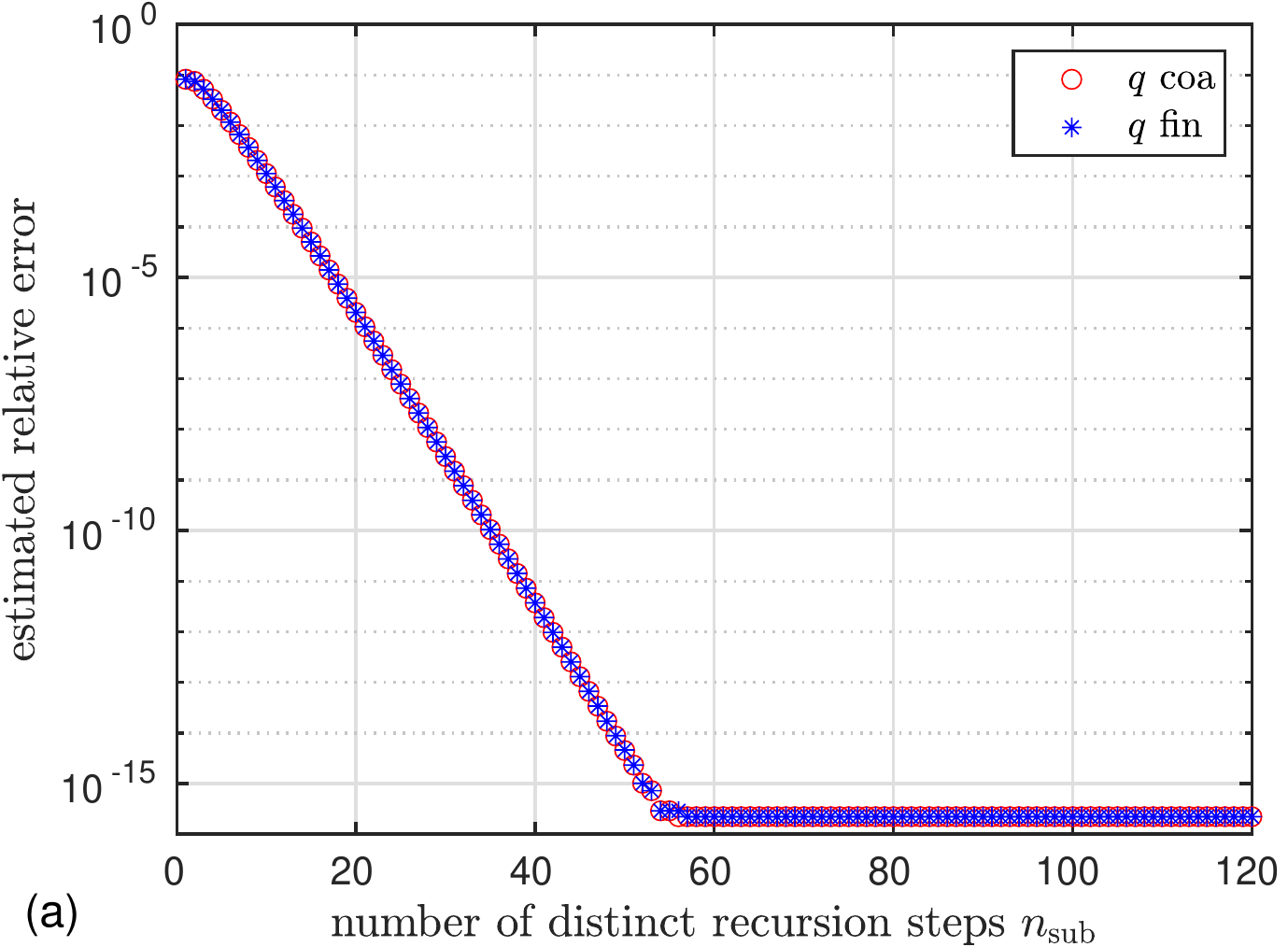}
   \includegraphics[height=42mm]{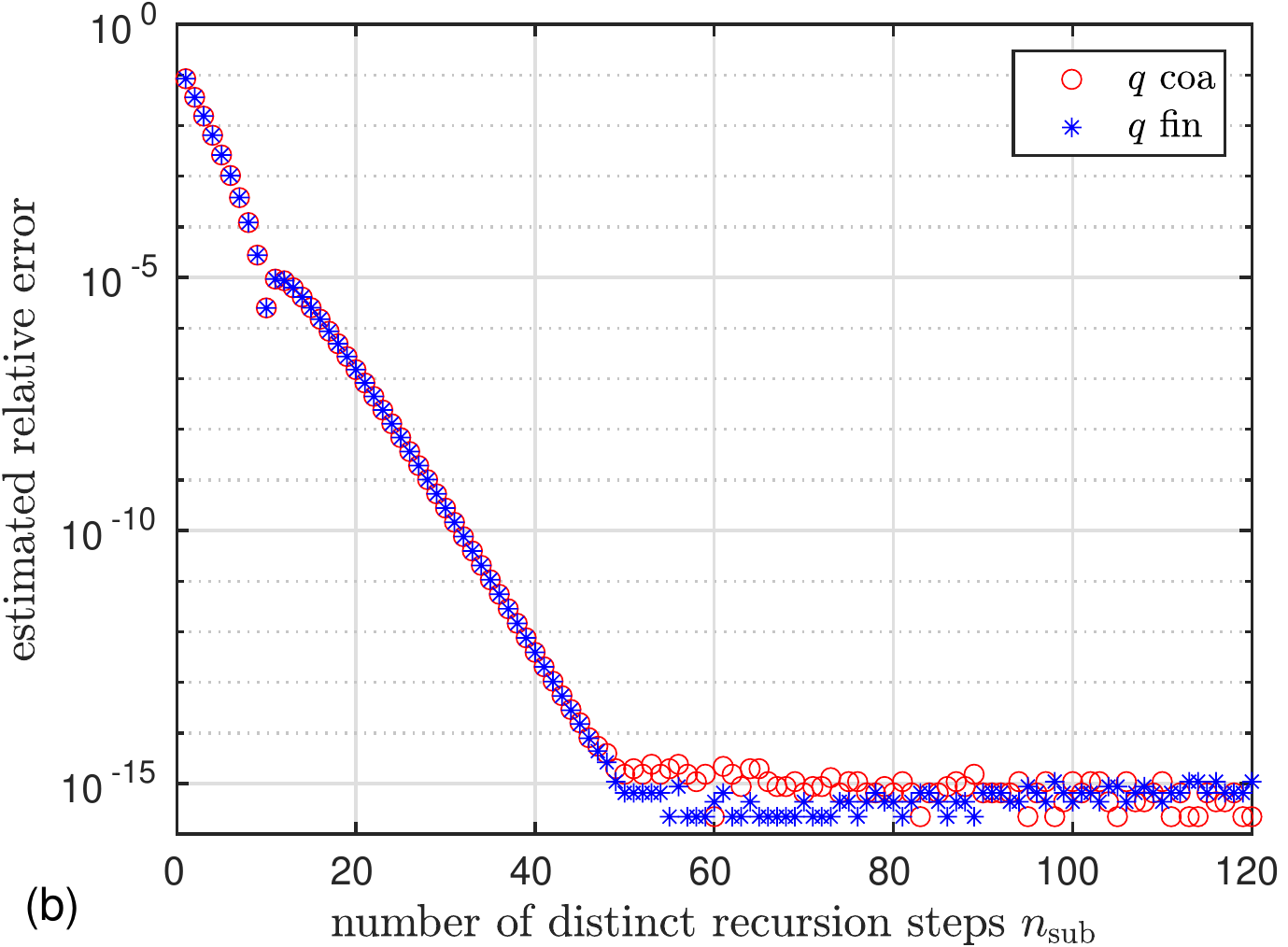}

   \vspace{4mm}

   \includegraphics[height=42mm]{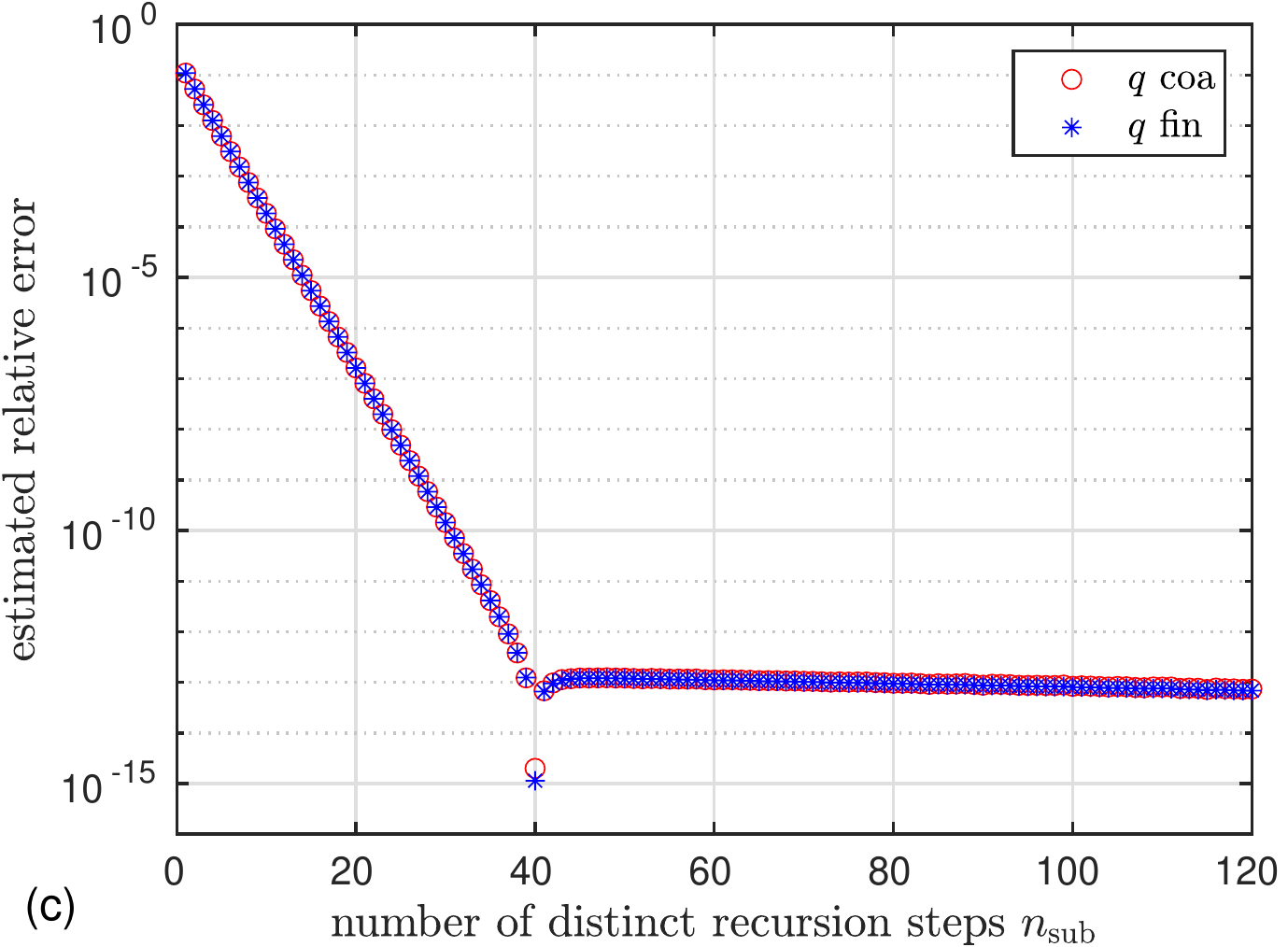}
   \includegraphics[height=42mm]{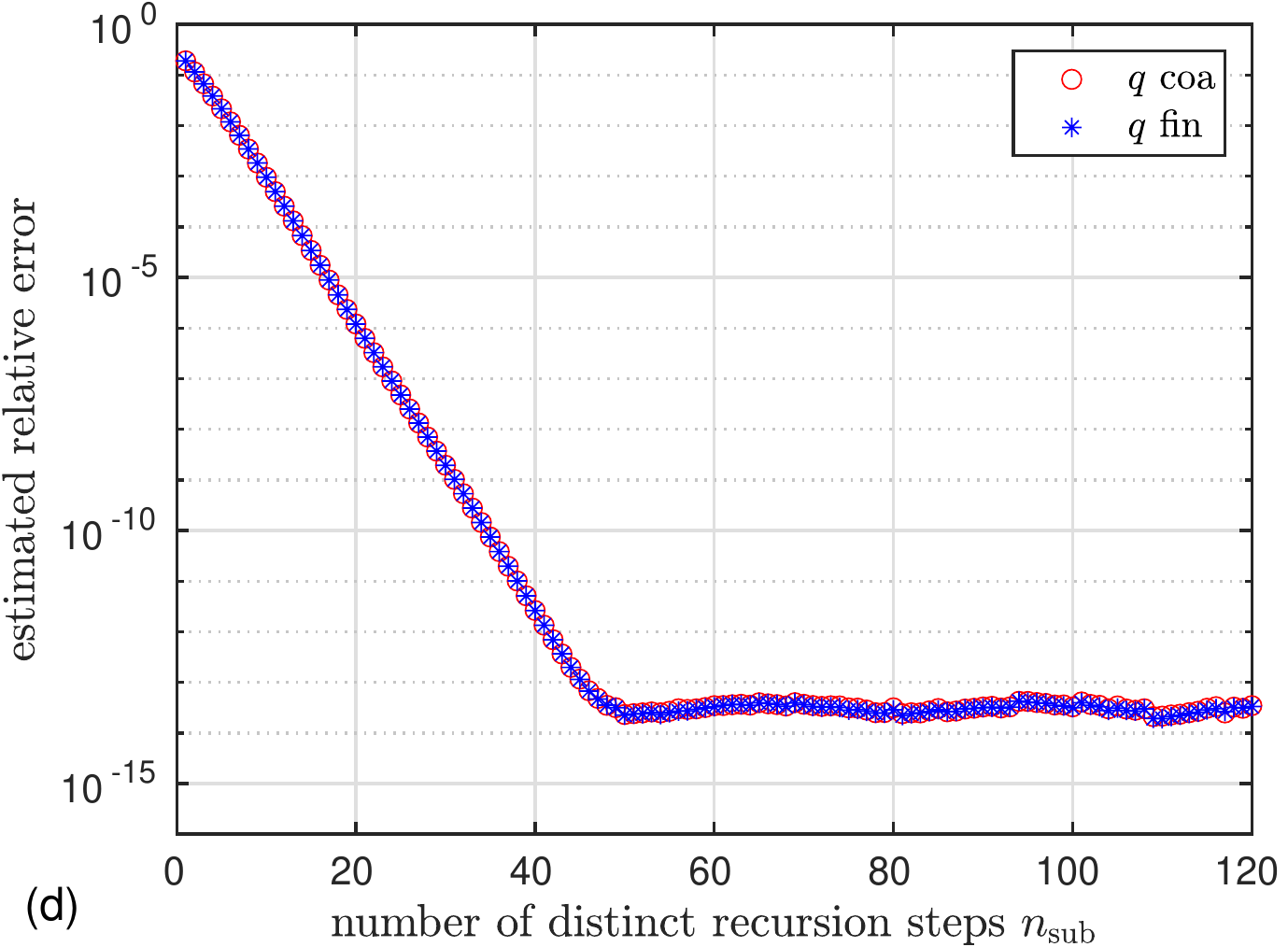}
\caption{\sf Convergence of $q$ with the number of distinct recursion
   levels $n_{\rm sub}$ in the example of Section~\ref{sec:ana}. The
   singularity strengths $\alpha$ of~(\ref{eq:fcirc}) are taken as: (a)
   $\alpha=0.5$; (b) $\alpha=0.94$; (c) $\alpha=0.99$; (d)
   $\alpha=1+0.3{\rm i}$.}
\label{fig:ana}
\end{figure}
Figure~\ref{fig:ana} shows results obtained with $10$ quadrature
panels on the coarse mesh on $\Gamma$, corresponding to $160$
discretization points on the coarse grid, and for various singularity
strengths $\alpha$. For $\alpha<1$ and not too close to $1$, the
results produced by our scheme are essentially fully accurate. Values
of $q$ computed on the coarse grid ($q_{\rm coa}$) and on the fine
grid ($q_{\rm fin}$) agree completely -- indicating that the
reconstruction procedure of Section~\ref{sec:backrecur} is stable.

Note, in Figure~\ref{fig:ana}(d), that for $\alpha=1+0.3{\rm i}$,
which corresponds to a singular right-hand side $f$ that is not even
in $L^1$, the scheme loses only about two digits of accuracy.  Note
also that, thanks to the use of initializers, a number $n_{\rm sub}=60$
of distinct recursion steps is more than enough to make $q$
converge (to the achievable precision)  for all values of $\alpha$
tested in Figure~\ref{fig:ana}. 

\subsubsection{A one-corner example}
\label{sec:snow}

We now consider an example where the corner opening angle
in~(\ref{eq:gamma}) is set to $\theta=\pi/2$. The contour $\Gamma$
then assumes the shape shown in Figure~\ref{fig:sno2}. The right-hand
side of~(\ref{eq:start}) is set to
\begin{equation}
f(r)=\frac{1}{\lvert r\rvert^\alpha}+\log{\lvert r\rvert}\,,
\label{eq:fsno}
\quad r\in\Gamma\,.
\end{equation}

\begin{figure}[t]
\centering 
  \includegraphics[height=42mm]{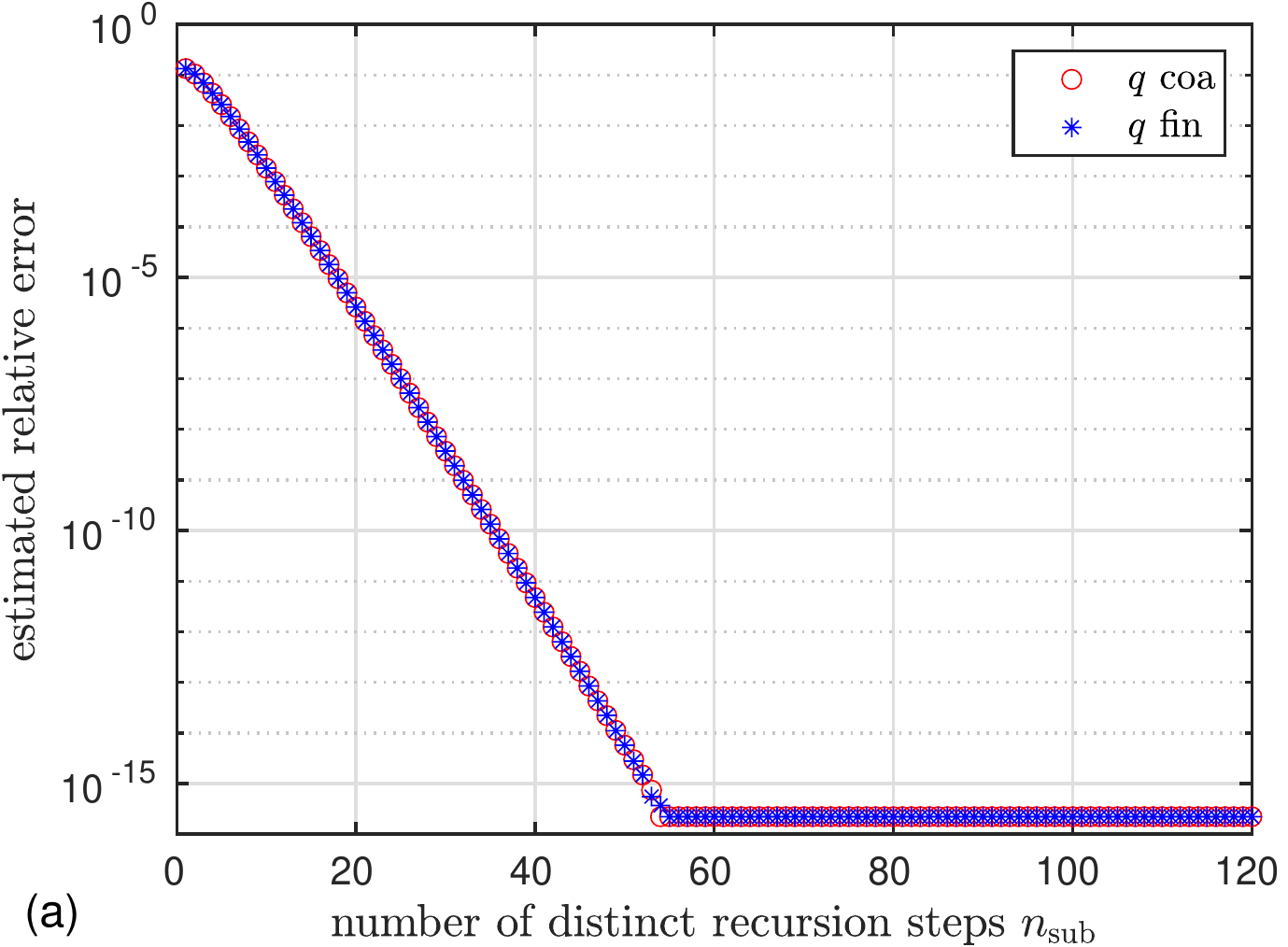}
  \includegraphics[height=42mm]{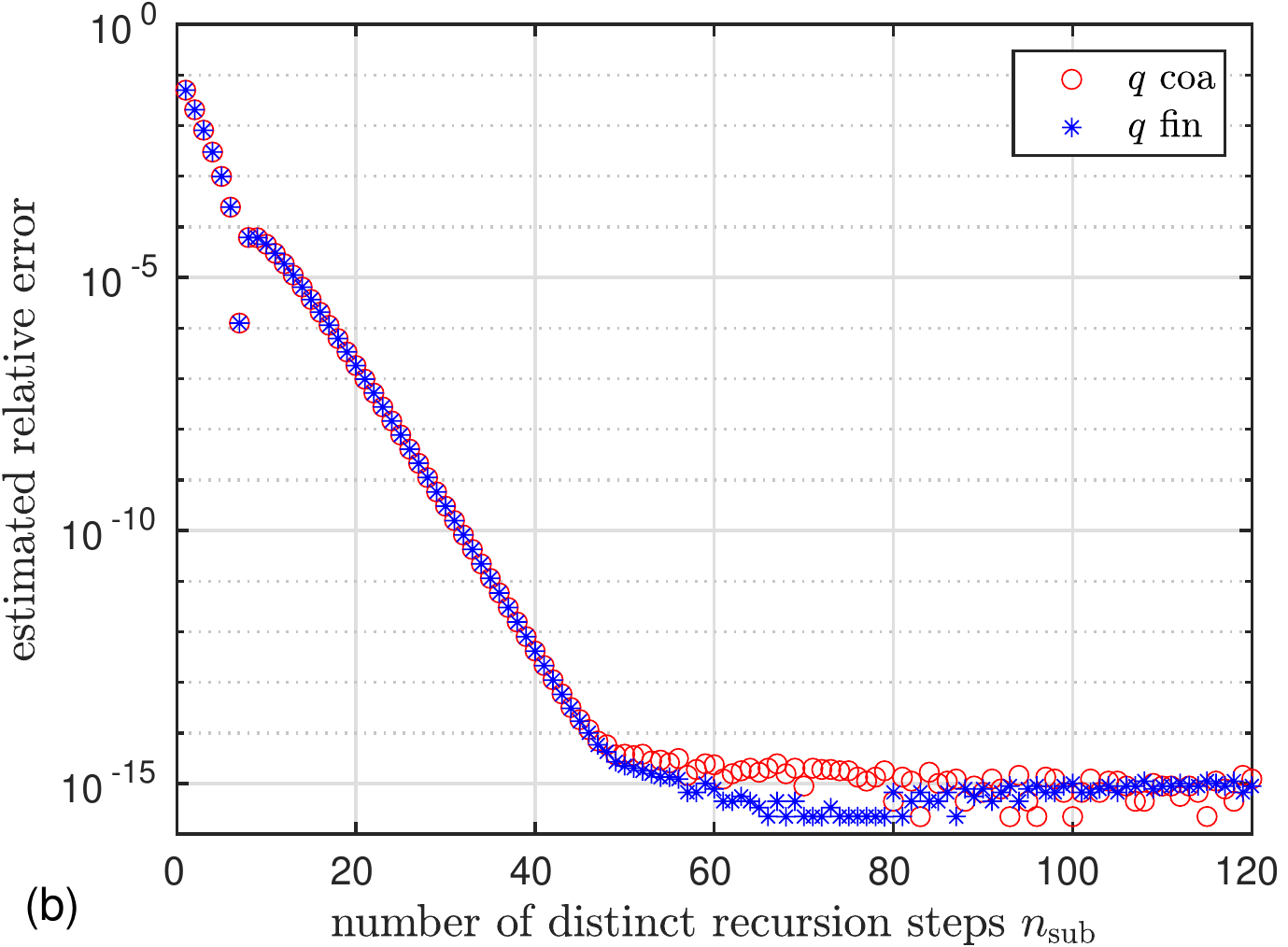}

  \vspace{4mm}
  
  \includegraphics[height=42mm]{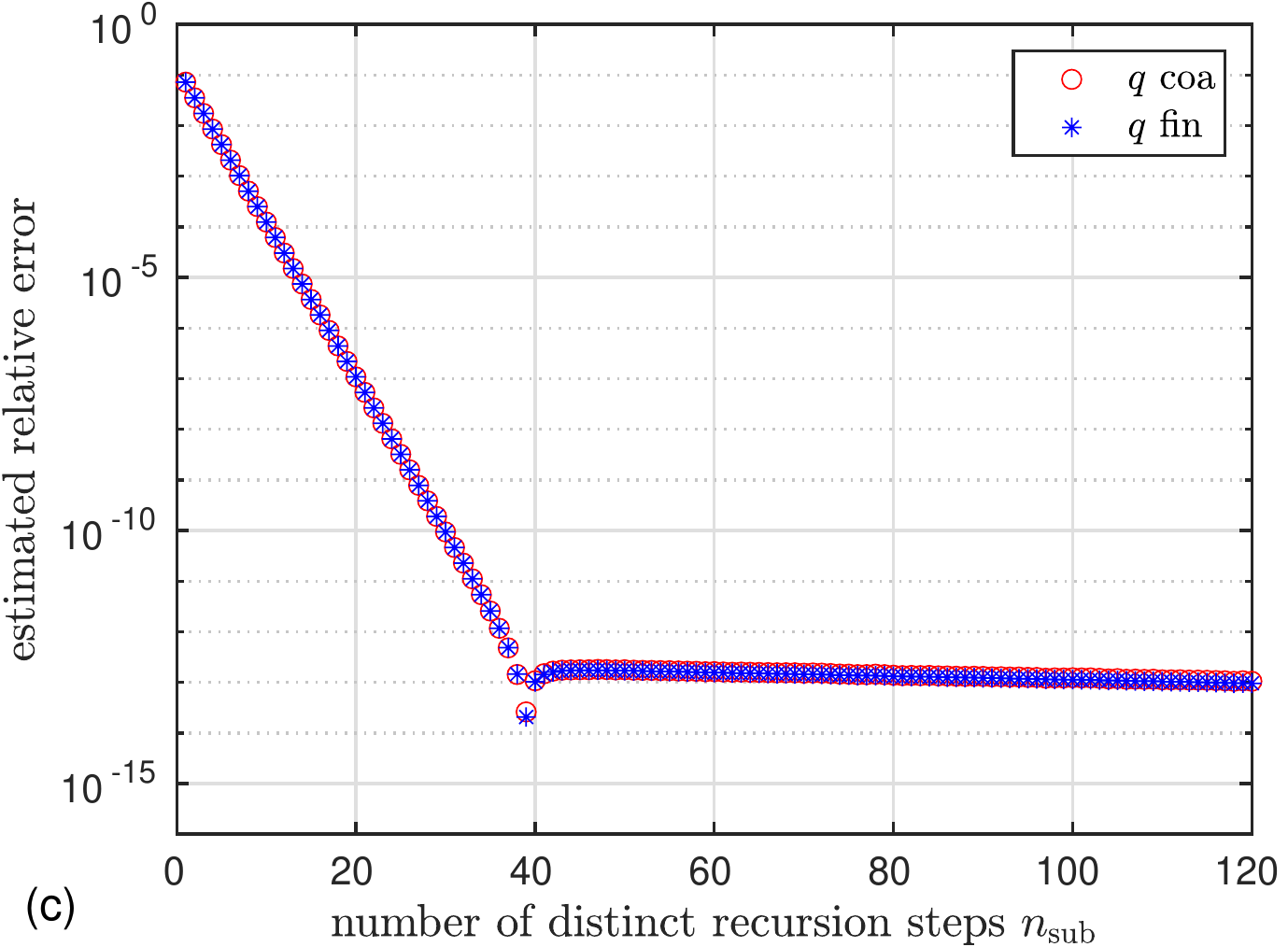}
  \includegraphics[height=42mm]{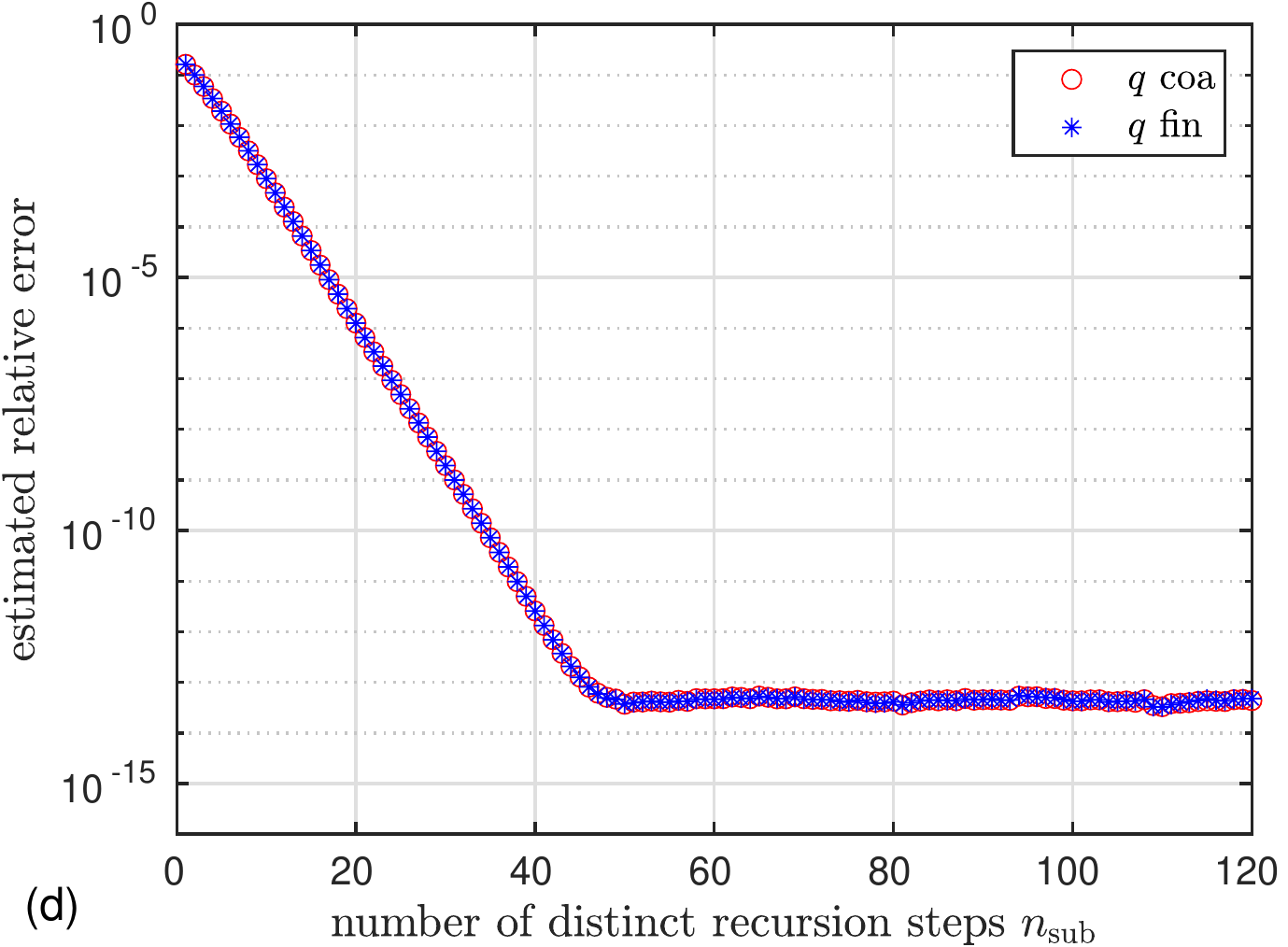}
\caption{\sf Convergence of $q$ with the number of distinct recursion 
  levels $n_{\rm sub}$ in the example of Section~\ref{sec:snow}. The
  singularity strengths $\alpha$ of~(\ref{eq:fsno}) are taken as: (a)
  $\alpha=0.5$; (b) $\alpha=0.94$; (c) $\alpha=0.99$; (d)
  $\alpha=1+0.3{\rm i}$.}
\label{fig:snow}
\end{figure}

Figure~\ref{fig:snow}, analogous to Figure~\ref{fig:ana}, shows
results obtained with $160$ discretization points on the coarse grid
on $\Gamma$ and for various singularity strengths $\alpha$. In the
absence of an analytical expression for $q$ we use, as a reference,
the value of $q_{\rm coa}$ obtained with $n_{\rm sub}=500$. For
example, with $\alpha=0.94$ this value is $q\approx 63.53529437281905$
and could be compared to the value $q\approx 63.53529437281894$,
obtained with $29,\!088$ discretization points on the fine grid on
$\Gamma$ using the $L^1$-norm-preserving Nystr{\" o}m discretization of
Askham and Greengard~\cite[Section 4]{AskhGree14}, which extends the
$L^2$-inner-product-preserving discretization of Bremer~\cite[Section
2]{Brem12}, in combination with compensated
summation~\cite{High96,Kaha65}. Since the relative difference between
these two reference values is on the order of our own error estimate
in Figure~\ref{fig:snow}(b), we believe that our error estimates are
reliable also in Figure~\ref{fig:snow}(c,d), where the norm-preserving
discretization cannot be used due to memory constraints.

We add that the execution of our code is very rapid. The computing
time, per data point, in Figure~\ref{fig:snow} varies with
$n_{\rm sub}$ and also differs between $q_{\rm coa}$ and $q_{\rm fin}$, but
it is much less than a second for all the data points shown.

\subsection{The exterior Dirichlet Helmholtz problem}

The exterior Dirichlet problem for the Helmholtz equation has been
considered in \cite[Section~18]{Hels18} in detail. Here we use this
problem to illustrate that its solution $U(r)$ can be found accurately
in the entire computational domain also when the right-hand side
$f(r)$ is singular or nearly singular. The integral representation of
$U(r)$ and the resulting boundary integral equation (a combined field
integral equation) are identical to those in
\cite[Section~18]{Hels18}. The boundary $\Gamma$ is given
by~(\ref{eq:gamma}) with $\theta=\pi/2$. We consider two cases

\begin{enumerate}[(a)]
\item {\it Singular right-hand side}\\
  The exact solution is
  \begin{equation}
  U(r)=H_0^{(1)}(\omega|r|)
  \end{equation}
  with $f(r)$ being the restriction of $U(r)$ to $\Gamma$. Here $\omega$
  is the wavenumber and $H_0^{(1)}$ is the first-kind Hankel function of
  order zero, that is, the field is generated by an acoustic monopole
  right at the corner.
\item {\it Nearly singular right-hand side}\\
  The exact solution is
  \begin{equation}
  U(r)=H_1^{(1)}(\omega|r-r'|) \frac{x-x'}{|r-r'|}\,,
  \end{equation}
  with $f(r)$ being the restriction of $U(r)$ to $\Gamma$, and
  $r'=(10^{-10},0)$. That is, the field is generated by an acoustic
  dipole very close to the corner.
\end{enumerate}

The integral equation \cite[Eq.~(68)]{Hels18} is solved with
$\omega=10$ using the method in this paper, the singular point is
taken as $\gamma=0$ both for (a) and for (b). For both cases,
the coarse grid on $\Gamma$ has 56 panels, i.e., 896 discretization points;
the number of refinement levels is set to $\nsub=112$;
and the number of GMRES iterations needed is 18. The field is
evaluated using the scheme in \cite[Section~20]{Hels18}. The solve phase
takes about 4 seconds, while the field evaluation takes about 100 seconds,
which can be accelerated via the fast multipole method~\cite{cheng2006cm}
when necessary. A Cartesian grid of $200 \times 200$ equispaced points is
placed on the rectangle $[-0.1, 1.1] \times [-0.53, 0.53]$ and evaluations
are carried out at those 27,760 grid points that are in the exterior
domain. For case (a), 5,088 target points activate local panelwise
evaluation for close panels; 920 target points close to the corner vertex
require that the solution $\boldsymbol{\rho}_{\rm fin}$ on the fine grid
is reconstructed using the backward recursion in Section~\ref{sec:backrecur}.
For case (b), the numbers are 6,364 and 1,326, respectively. 

 Figure~\ref{fig:ex2} shows the absolute error in the numerical
solution, demonstrating that our scheme achieves high accuracy in the
entire computational domain for both {\it singular} and {\it nearly
singular} right-hand sides.  The small difference in achievable
accuracy between (a) and (b) can chiefly be explained by that (b) has
a stronger singularity in the right-hand side $f(r)$ than has (a) and
is, thus, harder to resolve. 

\begin{figure}[t]
\centering
\includegraphics[height=43mm]{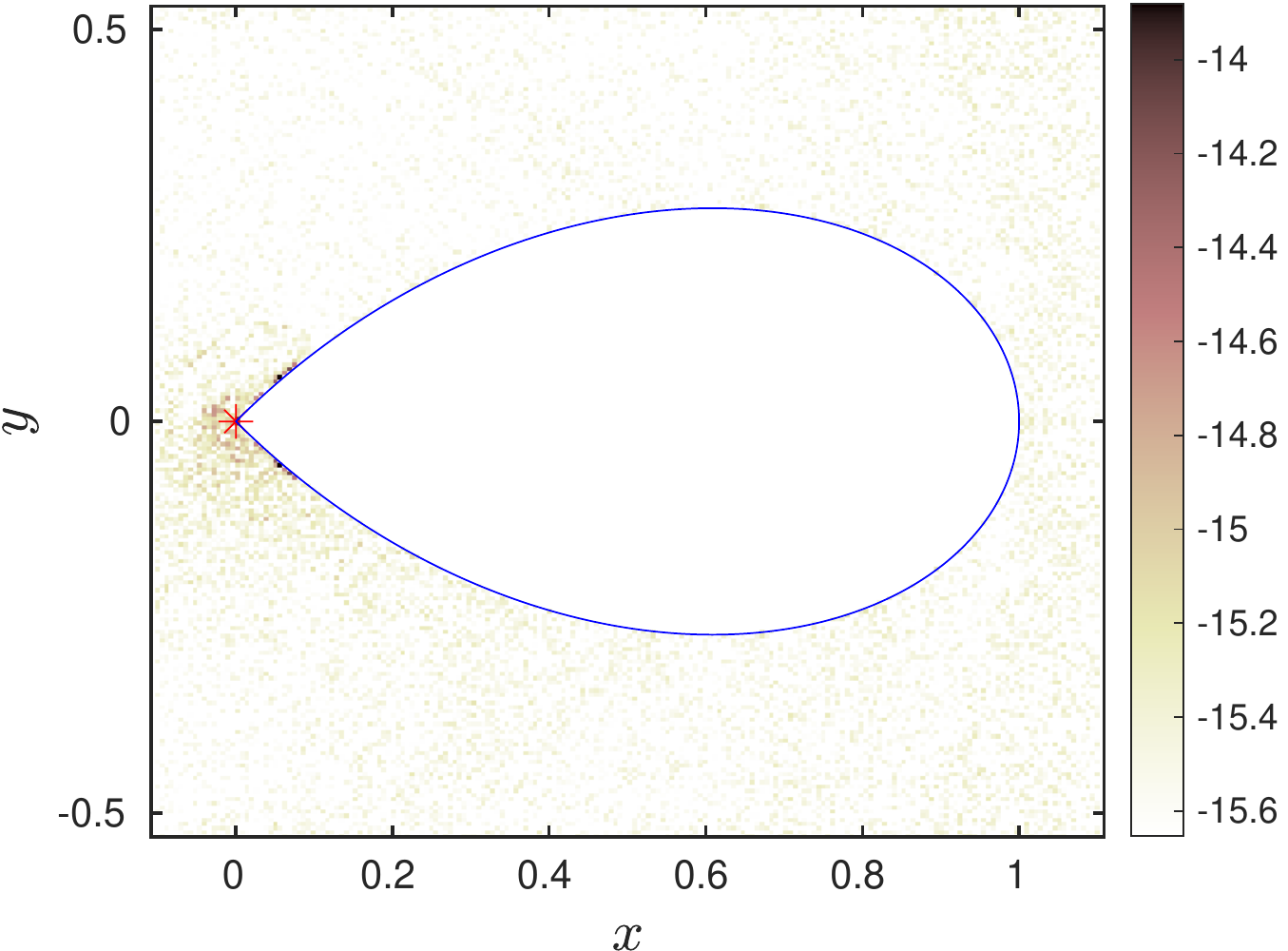}
\includegraphics[height=43mm]{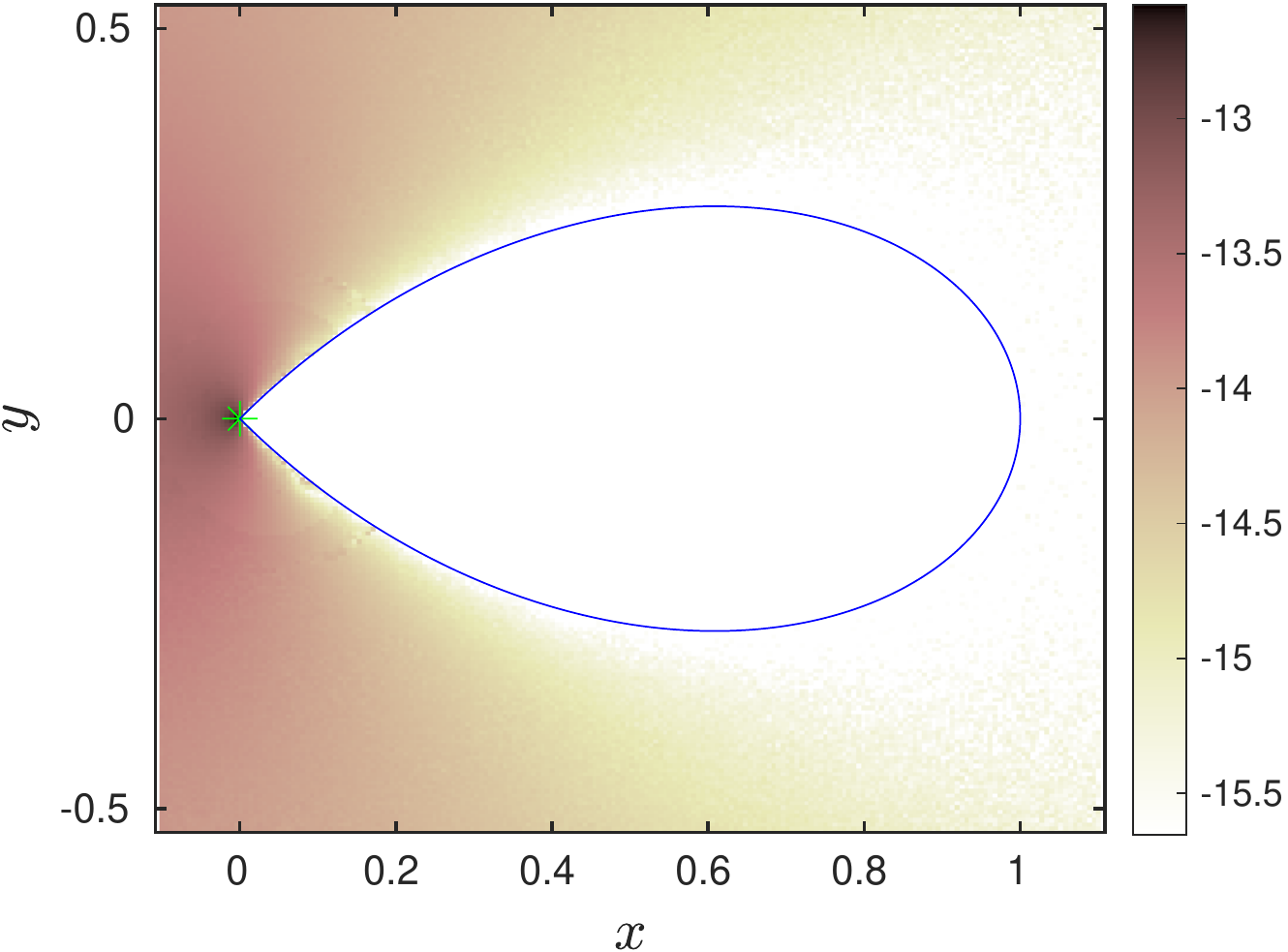}
\caption{\sf $\log_{10}$ of absolute error in the solution $U(r)$ to
   the exterior Dirichlet Helmholtz problem. 
   The blue curve is the boundary $\Gamma$. Left:
   $U(r)$ is an acoustic monopole field with the source at $r'=(0,0)$
   shown as a red star. Right: $U(r)$ is an acoustic dipole field with
   the source at $r'=(10^{-10},0)$ shown as a green star.}
\label{fig:ex2}
\end{figure}

\section{Application to the linearized BGKW equation for the Couette flow}
\begin{table}[t]
  \caption{\sf The velocity $u$ at $x=0.5$ at sample values of the Knudsen number $k$.
    Note that in~\cite{Jiang16}, the interval is shifted from $[-0.5,0.5]$ to $[0,1]$.
    Thus, the velocity $u$ at $x=0.5$ is listed as $u$ at $x=1$ in
    Tables 1 and 2 in~\cite{Jiang16}.}
  \sisetup{tight-spacing=true,exponent-product = \cdot,table-format=3.15}
\centering
\begin{tabular}{
    S[table-format=3.1]
    S[scientific-notation = true,table-format=5.15e-1]
    S[scientific-notation = true,table-format=4.15e-1]
    S[scientific-notation = true,table-format=3.1e-2]}
  \toprule
  ${k}$  & ${u(0.5)}$ { in \cite{Jiang16}} & ${u(0.5)}$ { (current)} & {Error}  \\
  \midrule
  0.003 &  4.978915352789693e-01 &    4.978915352789726e-01  &   6.6e-15 \\
  0.01  &  4.930697807742208e-01 &    4.930697807742217e-01  &   1.8e-15 \\
  0.03  &  4.800058682766829e-01 &    4.800058682766837e-01  &   1.6e-15 \\
  0.1   &  4.412246409722421e-01 &    4.412246409722424e-01  &   6.3e-16 \\
  0.3   &  3.672125695500504e-01 &    3.672125695500499e-01  &   1.4e-15 \\
  1.0   &  2.518613399894732e-01 &    2.518613399894736e-01  &   1.5e-15 \\
  2.0   &  1.852462993740218e-01 &    1.852462993740218e-01  &   1.5e-16 \\
  3.0   &  1.504282444992075e-01 &    1.504282444992074e-01  &   3.7e-16 \\
  5.0   &  1.126351880294592e-01 &    1.126351880294592e-01  &   2.5e-16 \\
  7.0   &  9.171689613521435e-02 &    9.171689613521428e-02  &   7.6e-16 \\
 10.0   &  7.292211299328497e-02 &    7.292211299328491e-02  &   7.6e-16 \\
 30.0   &  3.381357342231838e-02 &    3.381357342231840e-02  &   6.2e-16 \\
 100.0  &  1.343072948081877e-02 &    1.343072948081874e-02  &   1.9e-15 \\
 \bottomrule
 \end{tabular}
\label{table1}
\end{table}

In this section, we revisit the integral equation that is derived
from the linearized BGKW equation for the steady Couette
flow~\cite{Cercignani2000,Jiang16,Li2015}
\begin{subequations}
\begin{align}
  &
  u(x) 
  - 
  \frac{1}{k\sqrt{\pi}}
  \int_{-0.5}^{0.5} J_{-1} \left(\frac{|x - y|}{k} \right) 
  u(y) \,{\rm d}y
  =
  f(x)\,,\quad x\in[-0.5,0.5]\,,
\label{bgkw}
  \\
  &
  f(x) 
  =
  \frac{1}{2\sqrt{\pi}}
  \left[ 
    J_0\left(\frac{0.5-x}{k}\right)
    -
    J_0\left(\frac{0.5+x}{k}\right)
  \right],
\label{bgkwrhs}
\end{align}
\end{subequations}
where the parameter $k$ is the Knudsen number and $J_n$ is the $n$th
order Abramowitz function defined by
\begin{equation}
  J_n(x)
  =
  \int_0^\infty t^n e^{-t^2-x/t}\,{\rm d}t\,,
  \quad
  n \geq -1\,.
\label{abram}
\end{equation}
See, for example, \cite{Gimbutas20} and references therein for the
properties of Abramowitz functions and an accurate numerical
scheme for their evaluation.
The BGKW equation has been studied in \cite{Jiang16}, where it is shown
that the solution of \eqref{bgkw} contains singular terms $(x \ln x)^n$
for $n \in \mathbb{N} := \{1,\, 2,\, \ldots\}$ at the endpoints.
It is known that the kernel function $J_{-1}(x)$ has both absolute value
and logarithmic singularities at $x=0$, and that the right-hand-side
function \eqref{bgkwrhs} has $x \ln x$ singularity at the endpoints.
Benchmark calculations have been carried out in \cite{Jiang16},
using dyadic refinement towards the endpoints to treat the
singularities of the solution and the right-hand-side function, and generalized
Gaussian quadrature \cite{Ma96} to treat the kernel singularities. 

\begin{table}[t]
  \caption{\sf The stress $P_{xy}$ at sample values of the Knudsen number $k$.}
  \sisetup{tight-spacing=true,exponent-product = \cdot,table-format=3.15}
\centering
\begin{tabular}{
    S[table-format=3.1]
    S[scientific-notation = true,table-format=5.15e-1]
    S[scientific-notation = true,table-format=4.15e-1]
    S[scientific-notation = true,table-format=3.1e-2]}
  \toprule
  ${k}$  & ${P_{xy}}$ { in \cite{Jiang16}} & ${P_{xy}}$ { (current)} & {Error}  \\
  \midrule
  0.003 &  -1.490909702131201e-03 &    -1.490909702131188e-03  &   8.7e-15 \\
  0.01  &  -4.900405009657547e-03 &    -4.900405009657524e-03  &   4.8e-15 \\
  0.03  &  -1.413798601526842e-02 &    -1.413798601526841e-02  &   4.9e-16 \\
  0.1   &  -4.155607782558620e-02 &    -4.155607782558619e-02  &   3.3e-16 \\
  0.3   &  -9.344983511356682e-02 &    -9.344983511356685e-02  &   3.0e-16 \\
  1.0   &  -1.694625753368226e-01 &    -1.694625753368225e-01  &   3.3e-16 \\
  2.0   &  -2.083322536749375e-01 &    -2.083322536749375e-01  &   0.0e+00 \\
  3.0   &  -2.266437497658084e-01 &    -2.266437497658084e-01  &   0.0e+00 \\
  5.0   &  -2.446632678455994e-01 &    -2.446632678455994e-01  &   0.0e+00 \\
  7.0   &  -2.536943539674479e-01 &    -2.536943539674479e-01  &   0.0e+00 \\
 10.0   &  -2.611624603488405e-01 &    -2.611624603488405e-01  &   0.0e+00 \\
 30.0   &  -2.743853873277227e-01 &    -2.743853873277228e-01  &   2.0e-16 \\
 100.0  &  -2.796682147138912e-01 &    -2.796682147138912e-01  &   2.0e-16 \\
 \bottomrule
 \end{tabular}
\label{table2}
\end{table}

\begin{table}[t]
  \caption{\sf The half-channel mass flow rate $Q$ at 
    sample values of the Knudsen number $k$.}
  \sisetup{tight-spacing=true,exponent-product = \cdot}
\centering
\begin{tabular}{S[table-format=3.1]
    % without table-format, each column will have huge white space on
    % the left side of numbers. Here 4.15e-1 means: 4 spaces in front
    % of the decimal sign; 15 spaces after the decimal sign; e needs
    % space; minus sign needs space; 1 digit for the exponent.
    S[scientific-notation = true,table-format=4.15e-1]
    S[scientific-notation = true,table-format=3.15e-1]
    S[scientific-notation = true,table-format=3.1e-2]}
  \toprule
  ${k}$  & ${Q}$ { in \cite{Jiang16}} & ${Q}$ { (current)} & {Error}  \\
  \midrule
  0.003 &  1.242445655299172e-01 &    1.242445655299167e-01 &    4.4e-15 \\
  0.01  &  1.225330275292623e-01 &    1.225330275292621e-01 &    1.8e-15 \\
  0.03  &  1.180147037188893e-01 &    1.180147037188893e-01 &    1.2e-16 \\
  0.1   &  1.057028408172292e-01 &    1.057028408172292e-01 &    2.6e-16 \\
  0.3   &  8.560111699820618e-02 &    8.560111699820613e-02 &    4.9e-16 \\
  1.0   &  5.804708735555459e-02 &    5.804708735555460e-02 &    2.4e-16 \\
  2.0   &  4.281659776113917e-02 &    4.281659776113918e-02 &    1.6e-16 \\
  3.0   &  3.489298506190833e-02 &    3.489298506190833e-02 &    2.0e-16 \\
  5.0   &  2.627042060967383e-02 &    2.627042060967383e-02 &    1.3e-16 \\
  7.0   &  2.147460412330841e-02 &    2.147460412330841e-02 &    1.6e-16 \\
 10.0   &  1.714449048590649e-02 &    1.714449048590649e-02 &    2.0e-16 \\
 30.0   &  8.043009085700258e-03 &    8.043009085700263e-03 &    6.5e-16 \\
 100.0  &  3.226757181742400e-03 &    3.226757181742397e-03 &    9.4e-16 \\
 \bottomrule
 \end{tabular}
\label{table3}
\end{table}

We have implemented the method of this paper, based
on~\eqref{eq:prec2}, to solve \eqref{bgkw}. We note that
both endpoints are singular points and there is only one side to each
singular point, since we are dealing with an open arc instead of a closed
contour. This leads to straightforward modifications in the method
and its implementation. Some implementation details are as follows.
First, the kernel-split
quadrature is applied to treat the kernel singularity: the correction
to the logarithmic singularity was done in~\cite{Hels09JCP}; the correction
on diagonal blocks for the absolute value singularity can be derived easily
in an identical manner; the splitting of the kernel into various parts is done
using either the series expansion of the Abramowitz function~\cite{abram53} or the
Chebyshev expansion for each part as in~\cite{macleod1992anm}. When the Knudsen
number is low, the kernel is sharply peaked at the origin.
A modified version of the upsampling scheme in~\cite{klinteberg19}
is used to resolve the sharp peak
of the kernel accurately and efficiently so that the coarse panels only need to resolve
the features of the solution. Here the modification is that
the exact centering in~\cite{klinteberg19} is not enforced for local adaptive
panels for each target. Instead, the so-called level-restricted
property~\cite{ethridge2001sisc},
i.e., sizes of adjacent panels can differ at most by a factor of 2,
is used to ensure the accuracy in the calculation of the integrals. We remark
that the cost of the upsampling scheme is $O(\log(1/k))$ as $k\rightarrow 0$,
the same as that in~\cite{klinteberg19}.

Second, it is observed numerically that the condition number of the integral
equation \eqref{bgkw} increases as $k$ decreases, reaching about $2\cdot 10^4$
for $k=0.003$. On the other hand, the solution $u(x)$ approaches the
asymptotic solution $u_{\rm asym}(x)=x$ as $k\rightarrow 0$. Thus, in order
to reduce the effect of the ill-conditioning of the integral equation
for small values of $k$, we write $u(x)=w(x)+x$ and solve the following
equation for $w(x)$ instead when $k\le 0.3$:
\begin{subequations}
\begin{align}
  &
  w(x) 
  - 
  \frac{1}{k\sqrt{\pi}}
  \int_{-0.5}^{0.5} J_{-1} \left(\frac{|x - y|}{k} \right) 
  w(y) \,{\rm d}y
  =
  h(x)\,,\quad x\in[-0.5,0.5]\,,
\label{bgkw2}
  \\
  &
  h(x) 
  =
  -\frac{k}{\sqrt{\pi}}
  \left[ 
    J_1\left(\frac{0.5-x}{k}\right)
    -
    J_1\left(\frac{0.5+x}{k}\right)
  \right].
\label{bgkwrhs2}
\end{align}
\end{subequations}
Since $w(x)$ is a small perturbation when $k$ is small, the inaccuracy in the calculation
of $w(x)$ has almost no effect on the overall accuracy of $u(x)$. This allows us
to achieve the machine precision for all physical quantities of interest
for a much wider range of the Knudsen number $k$ than reported in~\cite{Jiang16}.

We have repeated the calculations in \cite{Jiang16} using the current method.
For all values of $k$, the computational domain $[-0.5,0.5]$ is divided into
four coarse panels; $n_{\rm sub}$ is set to $41$; and the GMRES stopping criterion
is set to machine epsilon.
Tables~\ref{table1}--\ref{table3} list the values of the velocity $u$ at $x=0.5$,
the stress $P_{xy}$ \cite[Eq.~(43)]{Jiang16} and
the half-channel mass flow rate $Q$ \cite[Eq.~(42)]{Jiang16}
at sample values of the Knudsen number $k$,
where the second column contains the values in \cite{Jiang16}, the third
column contains the values using the current method, and the last column shows
the relative difference between these values. Here $u(0.5)$ is obtained via the backward
recursion in Section~\ref{sec:backrecur}, $P_{xy}$ and $Q$ are calculated using the
velocity on the coarse grid with the kernel-split quadrature applied to treat the
logarithmic singularity of $J_0(x)$ at $x=0$ in $P_{xy}$. 
It is clear that the numerical results agree with those in \cite{Jiang16} to
machine precision for all sample values of $k$. 

The current method is much more efficient as compared with the one used
in \cite{Jiang16}. For $k=0.003$, the computation takes about $0.05$ second
of {\rm CPU} time; and for $k=10$, the computational time becomes unmeasurable
with {\sc Matlab}'s {\sf cputime} command for a single run, i.e., less than
$0.01$ second.
In \cite{Jiang16}, the timing results are $62.8$
and $23.9$ seconds, respectively. The RCIP method eliminates the
dyadic refinement towards the endpoints in the solve phase, and the upsampling
scheme allows us to use such coarse panels that they only need to
capture the features of the solution. The combination of these two techniques
enables us to use the optimal number of discretization points for constructing
the system matrix. The full machine precision accuracy of the current method
completely removes the need of using quadruple precision arithmetic.
Finally, the preconditioner ${\bf R}$ also reduces the number of {\rm GMRES} iterations.
Indeed, the number of {\rm GMRES} iterations required is at most $6$ for all sample values
of $k$, whereas $430$ was reported in \cite{Jiang16} for $k=0.003$.
All of this results in a significant reduction in the computational
cost, i.e., a speedup of at least a factor of $1000$.

\section{Concluding remarks}

Finding efficient solvers for SKIEs on non-smooth boundaries is a
field with substantial recent activity. However, almost all existing
methods only work well for smooth right-hand sides and may also
involve a fair amount of operator-specific analysis and
precomputed quantities. In contrast, the method constructed in
this paper for singular and nearly singular right-hand sides computes
all intermediary quantities needed on-the-fly and only involves a bare
minimum of analysis.

The extension of the current method to multiple right-hand sides
can be carried out easily. The compressed inverse ${\bf R}$ is independent
on $f$ and needs only to be computed once, while the compressed
inverse ${\bf R}_f$ depends on $f$ and needs to be computed afresh for
each new $f$. The setup cost of ${\bf R}_f$ for each additional $f$
can, furthermore, be reduced since several of the matrices entering
into the forward recursion for ${\bf R}_f$ are independent of $f$ and
can be stored after they have been computed on-the-fly. A similar
situation holds in the backward recursions for $\rho$ on the fine grid
-- should that quantity be needed. 

In a certain sense, the work completes the RCIP method in two dimensions
since the splitting into a smooth part and a local singular part is carried out
on both sides of the integral equation. The RCIP method can be generalized to
three dimensions. In~\cite{helsing2013acha}, the RCIP method was extended to solve
a boundary integral equation on the surface of a cube. In~\cite{helsing2014jcp},
it was extended to solve boundary integral equations on axially symmetric surfaces.
In a recent talk~\cite{gimbutas2021siamam},
an RCIP-type scheme for discretizing corner and edge singularities
in three dimensions was discussed. For general surfaces in three dimensions,
the extension of the RCIP is more or less straightforward when the geometry admits
a local hierarchical discretization near the corner and edge singularities;
and the extension of the work in this paper should follow subsequently.

\section*{Acknowledgments}

J. Helsing was supported by the Swedish Research Council under contract
2015-03780. S. Jiang was supported in part by the United States National
Science Foundation under grant DMS-1720405.
The authors would like to thank Anders Karlsson at Lund
University, Alex Barnett and Leslie Greengard at the Flatiron Institute
for helpful discussions.

\appendix
\section{Derivation of the forward recursion formula \eqref{eq:finalRf} for ${\bf R}_f$}
\label{sec:rfderiv}
First, we prove the following lemma.
\begin{lemma}
  Suppose that ${\bf A}$ is an invertible $m\times m$ matrix, ${\bf B}$
  is an $n\times n$ matrix, ${\bf U}$ is an $n\times m$ matrix,
  and ${\bf V}$ is an $m\times n$ matrix with $m\ge n$. Suppose further that
  ${\bf U} {\bf A}^{-1}{\bf V}$ and ${\bf A}+{\bf V}{\bf B}{\bf U}$
  are both invertible. Then
  \bea
     &{\bf U}({\bf A}+{\bf V}{\bf B}{\bf U})^{-1}{\bf V} =
     \left(({\bf U} {\bf A}^{-1}{\bf V})^{-1}+{\bf B}\right)^{-1},\\
     \label{rectinversion}
%  \ee
%  \be
     &{\bf U}({\bf A}+{\bf V}{\bf B}{\bf U})^{-1} =
     \left(({\bf U} {\bf A}^{-1}{\bf V})^{-1}
     +{\bf B}\right)^{-1}({\bf U} {\bf A}^{-1}{\bf V})^{-1}{\bf U} {\bf A}^{-1}.
     \label{rectinversion2}
  \eea
\end{lemma}

\begin{proof}
Introduce a complex parameter $\lambda$ and consider
$({\bf A}+\lambda {\bf V}{\bf B}{\bf U})^{-1}$. When $\lambda$
is sufficiently small, the following Taylor expansion is
valid
\be
({\bf A}+\lambda {\bf V}{\bf B}{\bf U})^{-1}
={\bf A}^{-1}
-\lambda {\bf A}^{-1}{\bf V}{\bf B}{\bf U}{\bf A}^{-1}
+\lambda^2 {\bf A}^{-1}{\bf V}{\bf B}{\bf U}{\bf A}^{-1}{\bf V}{\bf B}{\bf U}{\bf A}^{-1}
-\ldots,
\label{abinv}
\ee
Mutiplying both sides of \eqref{abinv} with ${\bf U}$ from the left and
with ${\bf V}$ from the right and regrouping, we obtain
\be
\ba
{\bf U}({\bf A}+\lambda{\bf V}{\bf B}{\bf U})^{-1}{\bf V} &=
{\bf U} {\bf A}^{-1}{\bf V} - \lambda \left({\bf U} {\bf A}^{-1}{\bf V}\right){\bf B}
\left({\bf U} {\bf A}^{-1}{\bf V}\right)\\
&+\lambda^2\left({\bf U} {\bf A}^{-1}{\bf V}\right) {\bf B}
\left({\bf U} {\bf A}^{-1}{\bf V}\right)
{\bf B}\left({\bf U} {\bf A}^{-1}{\bf V}\right)-\ldots\\
&=\left(\left({\bf U} {\bf A}^{-1}{\bf V}\right)^{-1}+\lambda{\bf B}\right)^{-1},
\label{abinv2}
\ea
\ee
where the second equality follows from the application of the Taylor expansion
in the reverse order. By the cofactor formula of the matrix inverse and the
so-called big formula for the matrix determinant~\cite{strang1993}, both
sides of \eqref{abinv2} are rational functions of $\lambda$.
Since these two rational functions are equal in a small neighborhood of the origin
in the complex plane, they must be equal everywhere in the whole complex plane
by analytic continuation \cite{ahlfors1966}.
Setting $\lambda=1$ in \eqref{abinv2}, we establish \eqref{rectinversion} and the
invertibility of $({\bf U} {\bf A}^{-1}{\bf V})^{-1}+{\bf B}$ simultaneously.
\eqref{rectinversion2} can be proved in an almost identical manner.
\end{proof}
\begin{figure}[t]
\centering 
\includegraphics[height=50mm]{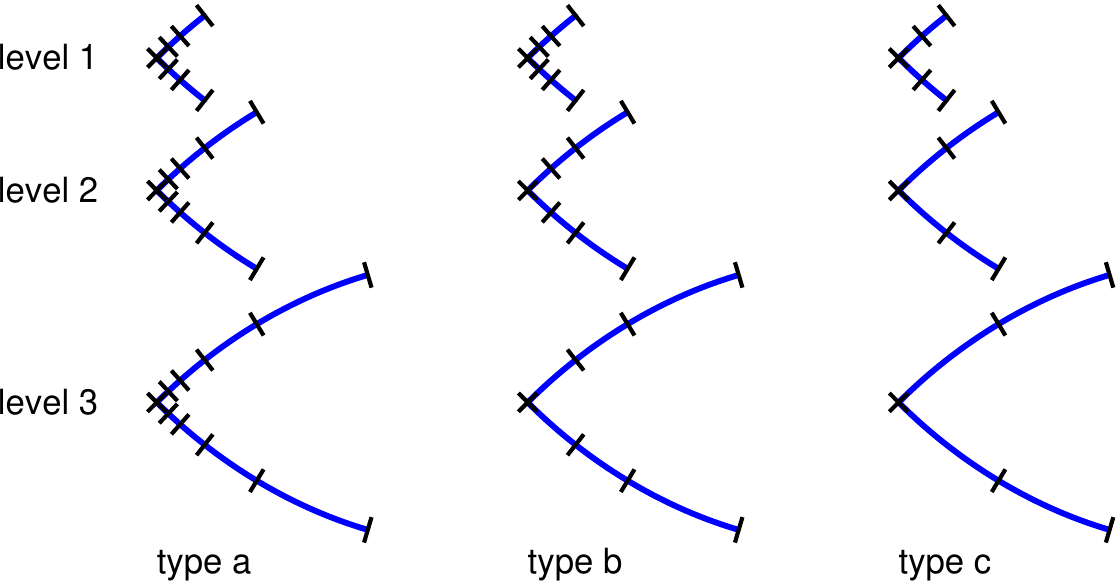}
\caption{\sf Meshes of type {\sf a}, type {\sf b}, and type {\sf c} at different levels.
  The type {\sf a} mesh contains $4+2i$ dyadic fine panels at level $i$.
  The type {\sf b} mesh always contains six panels, and the type {\sf c} mesh
  always contains four panels. At level $1$, the type {\sf a} mesh is identical
  to the type {\sf b} mesh. At level $\nsub$, the type {\sf c} mesh contains
  four coarse panels on $\Gamma^\star$ in Figure~\ref{fig:sno2}.}
\label{fig:3mesh}
\end{figure}
We now recall the definition of ${\bf R}_f$ in \eqref{eq:Rf}
\begin{equation}
{\bf R}_f=
{\bf P}_W^{\rm T}
\left({\bf I}_{\rm fin}+{\bf K}_{\rm fin}^\star\right)^{-1}
{\bf P}_f\,,
\label{rfdef}
\end{equation}
where ${\bf K}_{\rm fin}^\star$ is the system matrix built on a fine mesh
that is obtained via $\nsub$ level of dyadic refinement of the four
coarsest panels (with two panels on each side) near the singular point.
Let $\ngl$ be the number of Gauss-Legendre nodes on each panel. Then
the total number of discretization points on the fine mesh is $\ngl (4+2\nsub)$.
Clearly, direct application of $\eqref{rfdef}$ is very expensive, inaccurate,
and unrobust. Instead, the forward recursion is used to compute
$\bR_f$. The forward recursion starts from the finest six panels around
the singular point at level $i=1$, adds one panel on each side as the level
goes up, and reaches the full fine mesh at level $i=\nsub$. See Figure~\ref{fig:3mesh}
for an illustration of three different types of meshes at different levels,
where the type {\sf a} mesh is needed only in the derivation of the forward recursion.
The actual recursion formula involves only type {\sf b} and type {\sf c} meshes.
To be more precise, at any step in the forward recursion,
one only needs to build a system matrix of size $6\ngl \times 6\ngl$,
i.e., on the six panels of a type {\sf b} mesh, and the type {\sf c} mesh
is used only implicitly in the construction of the prolongation matrix
$\bP$ and its weighted version $\bP_W$. 
Similar to \cite[Eq.~(D.1)]{Hels18}, we define
\be
   {\bf R}_{fi}\coloneqq {\bf P}_{Wi{\rm ac}}^T
   ({\bf I}_{i{\rm a}}+{\bf K}_{i{\rm a}})^{-1}
   {\bf P}_{fi{\rm ac}}, \qquad i=1,\cdots, \nsub.
\ee
By the action of ${\bf P}_{fi{\rm ac}}$ and ${\bf P}_{Wi{\rm ac}}^T$, $\bR_{fi}$ is
always a matrix of size $4\ngl \times 4\ngl$ at any level $i$.
A derivation similar to that of \cite[Eq.~(D.6)]{Hels18} leads to
\be
{\bf R}_{fi}={\bf P}_{W{\rm bc}}^T{\bf P}_{Wi{\rm ab}}^T
({\bf I}_{i{\rm a}}+ \mathbb{F}\{{\bf K}_{(i-1){\rm a}}\}
+{\bf P}_{i{\rm ab}}{\bf K}_{i{\rm b}}^\circ{\bf P}_{Wi{\rm ab}}^T
)^{-1}
{\bf P}_{fi{\rm ab}}{\bf P}_{fi{\rm bc}},
\label{a6}
\ee
where we have assumed that the low-rank property
\be
\bK_{i{\rm a}}^\circ
={\bf P}_{i{\rm ab}}{\bf K}_{i{\rm b}}^\circ{\bf P}_{Wi{\rm ab}}^T
\label{lowrank}
\ee
holds to machine precision, as in \cite[Eq.~(D.3)]{Hels18}.
Since the type {\sf a} mesh differs from the type {\sf b} mesh
only on the two panels (i.e., $\Gamma_{i{\rm b}}^{\star\star}$)
closest to the singular point in the type {\sf b} mesh,
the only diagonal blocks in $\bP_{i{\rm ab}}$ and $\bP_{Wi{\rm ab}}$ 
that are not identity matrices
are $\bP_{i{\rm ab}}({\rm I}_{i{\rm a}}^{\star\star},{\rm I}_{i{\rm b}}^{\star\star})$
and $\bP_{Wi{\rm ab}}({\rm I}_{i{\rm a}}^{\star\star},{\rm I}_{i{\rm b}}^{\star\star})$,
respectively. Here ${\rm I}_{i{\rm a}}^{\star\star}$ and ${\rm I}_{i{\rm b}}^{\star\star}$
contain indices corresponding to discretization points in
$\Gamma_{i{\rm b}}^{\star\star}$ on type {\sf a} and type {\sf b} meshes,
respectively. Now it is straightforward to verify that \eqref{lowrank}
is equivalent to
\be
\ba
\bK_{i{\rm a}}({\rm I}_{i{\rm a}}^{\star\star},{\rm I}_{i{\rm b}}^{\circ})
&=\bP_{i{\rm ab}}({\rm I}_{i{\rm a}}^{\star\star},{\rm I}_{i{\rm b}}^{\star\star})
\bK_{i{\rm b}}({\rm I}_{i{\rm b}}^{\star\star},{\rm I}_{i{\rm b}}^{\circ})\,,\\
\bK_{i{\rm a}}({\rm I}_{i{\rm b}}^{\circ},{\rm I}_{i{\rm a}}^{\star\star})
&=\bK_{i{\rm b}}({\rm I}_{i{\rm b}}^{\circ},{\rm I}_{i{\rm b}}^{\star\star})
\bP_{Wi{\rm ab}}^T({\rm I}_{i{\rm b}}^{\star\star},{\rm I}_{i{\rm a}}^{\star\star})\,,
\label{lowrank2}
\ea
\ee
where ${\rm I}_{i{\rm b}}^{\circ}$ contains indices corresponding to
discretization points on $\Gamma_{i{\rm b}}^{\circ}$, that is, on the
two outermost panels from $\gamma$. Since $\Gamma_{i{\rm b}}^{\star\star}$
is always well-separated from $\Gamma_{i{\rm b}}^{\circ}$ for any
level $i$, \eqref{lowrank2} is equivalent to
stating that the kernel function $K(r,r')$ is smooth when
$r\in \Gamma_{i{\rm b}}^{\star\star}$ and $r'\in \Gamma_{i{\rm b}}^{\circ}$,
or vice versa, and thus the interaction between
$\Gamma_{i{\rm b}}^{\star\star}$ and $\Gamma_{i{\rm b}}^{\circ}$ can
be discretized to machine precision with $\ngl$ points on each panel
in $\Gamma_{i{\rm b}}^{\star\star}$ provided that $\ngl$ is not too
small. This holds for all kernels we have encountered in practice,
including, say, highly oscillatory ones, if the coarse mesh is chosen
in such a way that the oscillations of the kernel are well-resolved.
We emphasize that this is a property of the kernel function $K(r,r')$
and the type {\sf a} and {\sf b} meshes. It has nothing to do with
the singularity of the right-hand side $f$.

Applying \eqref{rectinversion2} to part of the right-hand side
of \eqref{a6}, we obtain
\be
\ba
&{\bf P}_{Wi{\rm ab}}^T
({\bf I}_{i{\rm a}}+ \mathbb{F}\{{\bf K}_{(i-1){\rm a}}\}+{\bf P}_{i{\rm ab}}
{\bf K}_{i{\rm b}}^\circ{\bf P}_{Wi{\rm ab}}^T)^{-1}=\\
&\left[
  \left({\bf P}_{Wi{\rm ab}}^T({\bf I}_{i{\rm a}}+\mathbb{F}\{{\bf K}_{(i-1){\rm a}}\})^{-1}
  {\bf P}_{i{\rm ab}}\right)^{-1}
  +{\bf K}_{i{\rm b}}^\circ\right]^{-1}\\
& \cdot \left({\bf P}_{Wi{\rm ab}}^T({\bf I}_{i{\rm a}}+\mathbb{F}\{{\bf K}_{(i-1){\rm a}}\})^{-1}
  {\bf P}_{i{\rm ab}}\right)^{-1}
  {\bf P}_{Wi{\rm ab}}^T({\bf I}_{i{\rm a}}+\mathbb{F}\{{\bf K}_{(i-1){\rm a}}\})^{-1}.
\label{a7}
\ea
\ee
Recall that \cite[Eq.~(D.10)]{Hels18} states that
\be
\mathbb{F}\{{\bf R}_{i-1}\}+{\bf I}_{\rm b}^\circ =
{\bf P}_{Wi{\rm ab}}^T
({\bf I}_{i{\rm a}}+\mathbb{F}\{{\bf K}_{(i-1){\rm a}}\})^{-1}
{\bf P}_{i{\rm ab}}.
\label{a8}
\ee
\begin{figure}[t]
\centering 
\includegraphics[height=33mm]{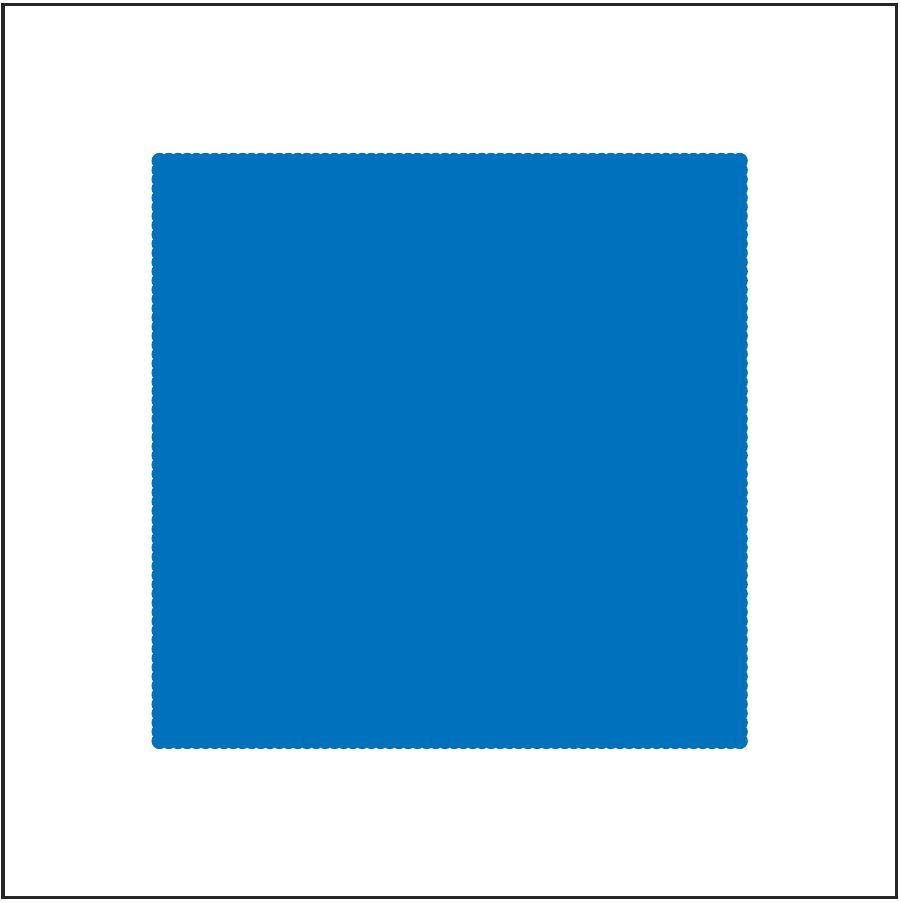}
\hspace{3mm}
\includegraphics[height=33mm]{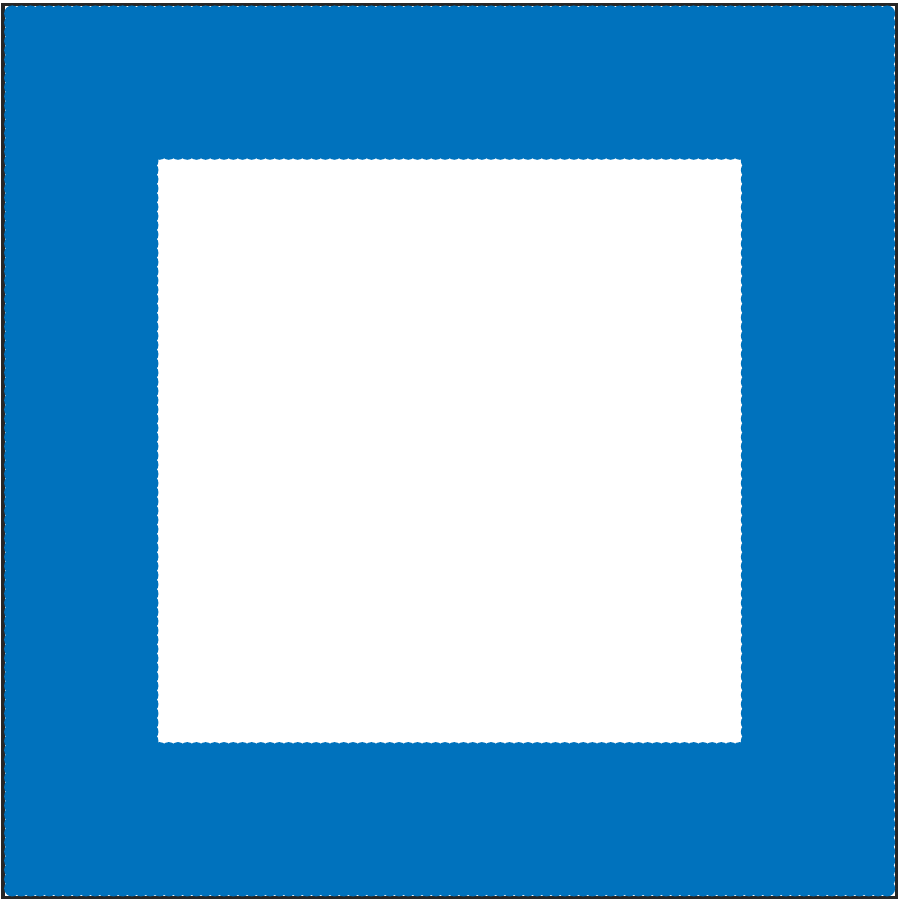}
\hspace{3mm}
\includegraphics[height=33mm]{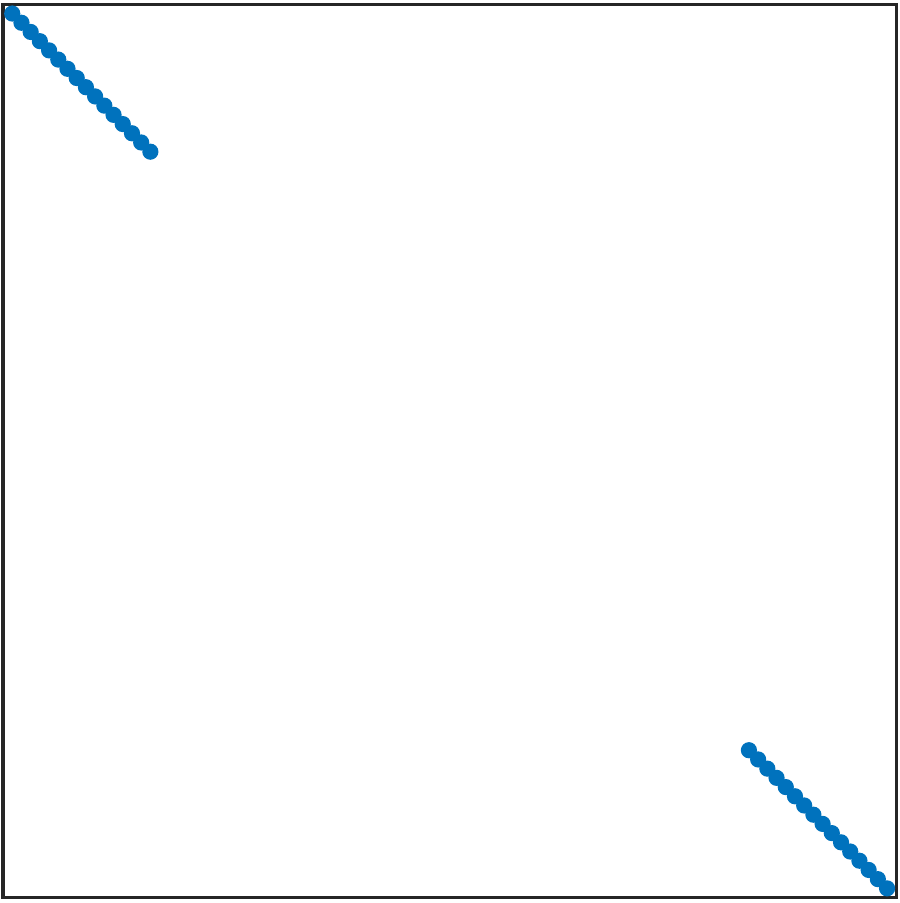}
\caption{\sf Nonzero patterns of $\mathbb{F}\{{\bf R}_{i-1}^{-1}\}$
  and $\mathbb{F}\{{\bf R}_{f(i-1)}\}$
  (left), ${\bf K}_{i{\rm b}}^\circ$ (center), and ${\bf I}_{\rm b}^\circ$ (right).
  Note that the patterns depend on the fact that when constructing
  the system matrix in the forward recursion,
  the sources and targets are arranged in the order from the top panel
  to the bottom panel on a type {\sf b} mesh, see Figure~\ref{fig:3mesh}.
}
\label{fig:sparsepattern}
\end{figure}
We observe that $\mathbb{F}\{{\bf R}_{i-1}\}$ places ${\bf R}_{i-1}$
in the center $4\ngl \times 4\ngl$ block with zero padding
in a $6\ngl\times 6\ngl$ matrix, and the nonzero blocks
of ${\bf I}_{\rm b}^\circ$ are the identity matrix of size
$\ngl$ in the top and bottom diagonal blocks (Figure~\ref{fig:sparsepattern}).
Thus,
\be
\ba
\mathbb{F}\{{\bf R}_{i-1}^{-1}\}+{\bf I}_{\rm b}^\circ &=
\left(\mathbb{F}\{{\bf R}_{i-1}\}+{\bf I}_{\rm b}^\circ\right)^{-1}\\
&=
\left({\bf P}_{Wi{\rm ab}}^T
({\bf I}_{i{\rm a}}+\mathbb{F}\{{\bf K}_{(i-1){\rm a}}\})^{-1}
{\bf P}_{i{\rm ab}}\right)^{-1}.
\label{a9}
\ea
\ee
Similarly,
\be
\ba
\mathbb{F}\{{\bf R}_{f(i-1)}\}&=
\mathbb{F}\{{\bf P}_{W(i-1){\rm ac}}^T({\bf I}_{(i-1){\rm a}}+{\bf K}_{(i-1){\rm a}})^{-1}
{\bf P}_{f(i-1){\rm ac}}\}\\
&=
{\bf P}_{Wi{\rm ab}}^T
({\bf I}_{i{\rm a}}+\mathbb{F}\{{\bf K}_{(i-1){\rm a}}\})^{-1}
{\bf P}_{fi{\rm ab}}
-{\bf I}_{\rm b}^\circ,
\label{a10}
\ea
\ee
where the second equality uses
$\left({\bf P}_{Wi{\rm ab}}^T{\bf P}_{fi{\rm ab}}\right)^\circ={\bf I}_{\rm b}^\circ$.
Combining \eqref{a6}, \eqref{a7}--\eqref{a10}, we obtain
\be
{\bf R}_{fi}={\bf P}_{W{\rm bc}}^T
\left(\mathbb{F}\{{\bf R}_{i-1}^{-1}\}+{\bf I}_{\rm b}^\circ+
{\bf K}_{i{\rm b}}^\circ\right)^{-1}
\left(\mathbb{F}\{{\bf R}_{i-1}^{-1}\}+{\bf I}_{\rm b}^\circ\right)
\left(\mathbb{F}\{{\bf R}_{f(i-1)}\}+{\bf I}_{\rm b}^\circ\right)
{\bf P}_{fi{\rm bc}}.
\label{a11}
\ee
Finally, we arrive at \eqref{eq:finalRf} by observing that
\be
\left(\mathbb{F}\{{\bf R}_{i-1}^{-1}\}+{\bf I}_{\rm b}^\circ\right)
\left(\mathbb{F}\{{\bf R}_{f(i-1)}\}+{\bf I}_{\rm b}^\circ\right)
=\mathbb{F}\{{\bf R}_{i-1}^{-1}{\bf R}_{f(i-1)}\}+{\bf I}_{\rm b}^\circ,
\ee
due to the nonzero patterns of the involved matrices
shown in Figure~\ref{fig:sparsepattern}.

\section{A synopsis of the RCIP demo codes}
\label{sec:Frecursiondetail}
On the website \url{http://www.maths.lth.se/na/staff/helsing/Tutor/},
maintained by the first author, a set of {\sc Matlab} demo codes are
posted so that researchers who are interested in the RCIP method can
download and run them and gain first-hand experience about the
method's efficiency, accuracy, and robustness for a few chosen
problems. The codes are written so that they can easily be modified to
solve other problems. The reader is expected to read the tutorial~\cite{Hels18}
in tandem with looking at the codes. Many demo codes use the
one-corner contour $\Gamma$ in Figure~\ref{fig:sno2}, which is
parameterized by $s\in [0,1]$ with $s=0$ corresponding to the corner
at the origin, and going back to the corner again in the
counterclockwise direction as $s$ increases to $s=1$. Most demo codes
consist of the following building blocks or functions:
\begin{itemize}
\item Boundary discretization functions. These include the
   discretization in parameter space using panel division and functions
   returning $z(s)$, $z'(s)$, $z''(s)$ in complex form for a given
   parameter value $s$, e.g., {\tt zfunc}, {\tt zpfunc}, {\tt zppfunc},
   {\tt zinit}, {\tt panelinit}.
\item Integral operator discretization functions. These include
   splitting of the kernel into various parts -- smooth part,
   logarithmically singular part, Cauchy singular part, hypersingular
   part, etc., corrections for singular and nearly singular integrals
   (explicit kernel-split quadrature), e.g., {\tt LogCinit}, {\tt
     WfrakLinit}, {\tt wLCinit}, {\tt StargClose}, {\tt KtargClose}.
\item System matrix construction functions such as {\tt MAinit}, {\tt
     MRinit}, {\tt Soperinit}, {\tt Koperinit}.
\item Functions for the forward recursion formula for computing the
   nontrivial block of $\bR$. These include {\tt zlocinit}, {\tt
     Rcomp}, {\tt SchurBana}.
\item Other utility functions such as {\tt myGMRES}, {\tt Tinit16},
   {\tt Winit16} that are self-explanatory.
\end{itemize}

We now provide some comments on the forward recursion functions. In
the discussions below, $\ngl=16$, as in most demo codes. The function
{\tt zlocinit} returns a discretization on the type {\sf b} mesh at
level $i$, where the points are always arranged in the order from the
top panel to the bottom panel. Another important feature is that in
{\tt zlocinit}, the singular points $\gamma_j$ is translated to the
origin to reduce the effect of round-off error. See Figure~\ref{figb1}
for detailed comments on {\tt zlocinit}.

Recall from~(\ref{eq:finalR}) that the forward recursion formula for
$\bR$ is
\begin{equation}
{\bf R}_i={\bf P}^T_{W\rm{bc}}
\left(
\mathbb{F}\{{\bf R}_{i-1}^{-1}\}+{\bf I}_{\rm b}^\circ+{\bf K}_{i{\rm
     b}}^\circ\right)^{-1}{\bf P}_{\rm{bc}}\,,
\quad i=1,\ldots,n_{\rm sub}\,.
\label{eq:Rrecursion}
\end{equation}
Let $\bM_i\coloneqq {\bf I}_{\rm b}+{\bf K}_{i{\rm b}}$
and $\bM_i^\circ\coloneqq {\bf I}_{\rm b}^\circ+{\bf K}_{i{\rm b}}^\circ$.
Then Figure~\ref{fig:sparsepattern} shows that there is no overlap in
nonzero entries of $\bM_i^\circ$ and $\bR_{i-1}$. Direct implementation
of \eqref{eq:Rrecursion} is shown in Algorithm~\ref{alg1}.

\begin{algorithm}[H]
  \caption[Caption]{\label{alg1} Forward recursion via direct implementation
    of \eqref{eq:Rrecursion}.}
  \begin{algorithmic}[1]
    \For{$i=1,\ldots,\nsub$}    
    \State Obtain 96 discretization points on the type {\sf b} mesh at level $i$.
    \State Construct $96\times 96$ system matrix $\bM_i$ at level $i$.
    \If{$i==1$}
    \State $\bR_0=\bM_1^{\star-1}$. That is, $\bR_0$ is the inverse of the center
    $64\times 64$\\ \hspace{0.5in} block of $\bM_1$.
    \EndIf
    \State Compute $\bR_{i-1}^{-1}$.
    \State Replace the center $64\times 64$ block of $\bM_i$ with $\bR_{i-1}^{-1}$
    to obtain $\widetilde{\bM}_i$.
    \State Compute $\widetilde{\bM}_i^{-1}$.
    \State Compress $\widetilde{\bM}_i^{-1}$ to obtain $\bR_i$. To be more precise,
    $\bR_i={\bf P}^T_{W\rm{bc}}\widetilde{\bM}_i^{-1}{\bf P}_{\rm{bc}}$.
    \EndFor
    \State Return the nontrivial $64\times 64$ block of $\bR$, i.e., $\bR_{\nsub}$.
  \end{algorithmic}
\end{algorithm}

Algorithm~\ref{alg1} needs to invert two matrices, i.e.,
$\bR_{i-1}$ of size $64\times 64$ and $\widetilde{\bM}_i$ of size $96\times 96$.
This may lead to loss of accuracy and even instability when
$\bR_{i-1}$ is ill-conditioned. As pointed out in Section~\ref{sec:improvement},
the algorithm can be improved via the block matrix inversion formula that
avoids explicit inversion of $\bR_{i-1}$. Indeed, after column and row
permutations, we obtain the following block structure
(Figure~\ref{fig:Rpermutation})
\be
\mathbb{F}\{{\bf R}_{i-1}^{-1}\}+{\bf I}_{\rm b}^\circ+{\bf K}_{i{\rm
     b}}^\circ \longrightarrow
\left[
\begin{array}{c c}
      \bA^{-1}
& \bU \\
   \bV & \bD
  \end{array}
\right]\,.
\ee
\begin{figure}[t]
\centering
\includegraphics[height=33mm]{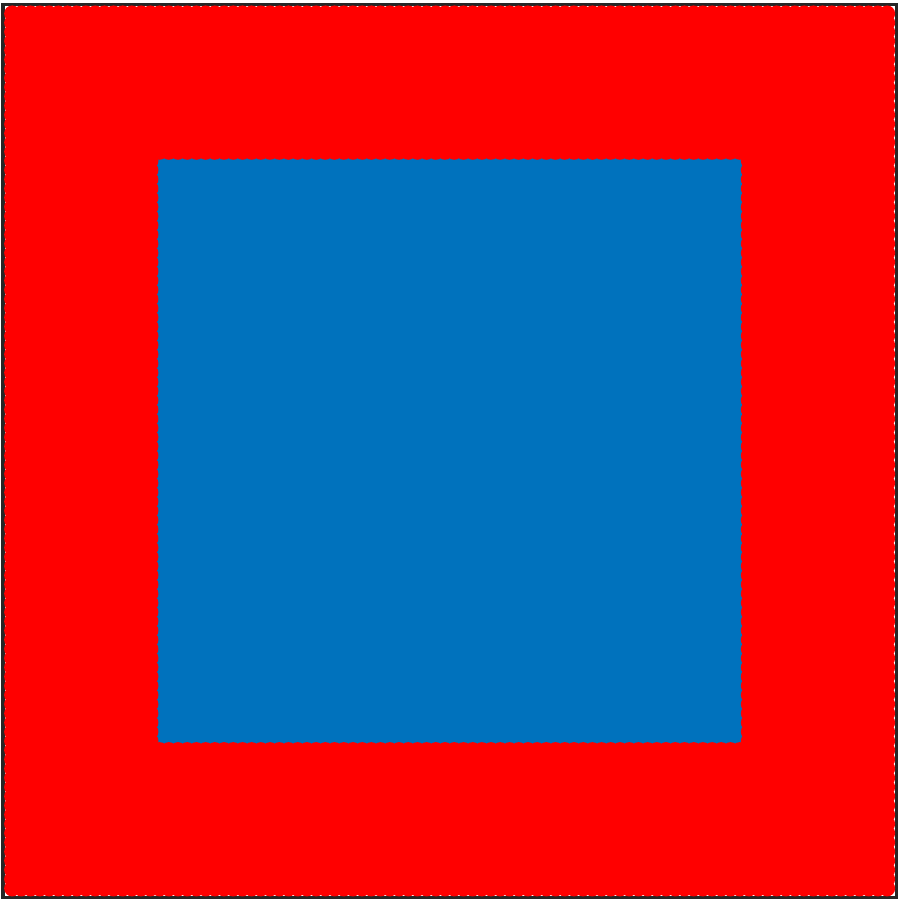}
\hspace{10mm}
\includegraphics[height=33mm]{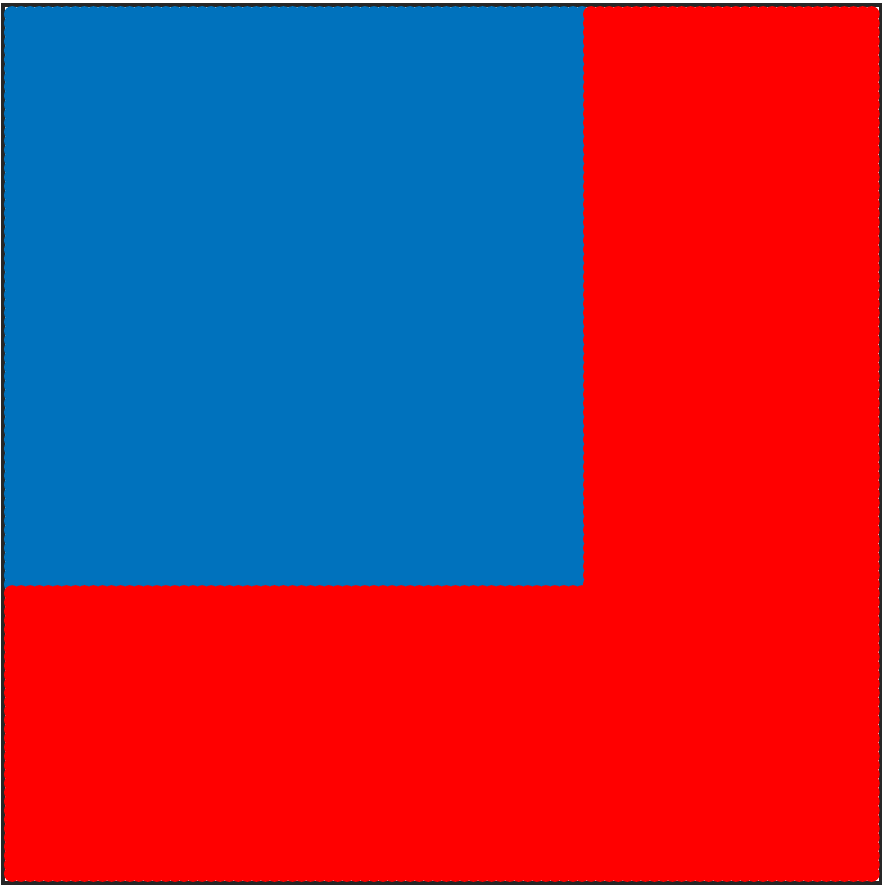}
\caption{\sf Column/row permutation converting
   $\mathbb{F}\{{\bf R}_{i-1}^{-1}\}+{\bf I}_{\rm b}^\circ+{\bf
     K}_{i{\rm b}}^\circ$ into a $2\times 2$ block matrix. Blue block
   is $\bR_{i-1}^{-1}$ and red block is ${\bf I}_{\rm b}^\circ+{\bf
     K}_{i{\rm b}}^\circ$. Left: the original matrix $\mathbb{F}\{{\bf
     R}_{i-1}^{-1}\}+{\bf I}_{\rm b}^\circ+{\bf K}_{i{\rm b}}^\circ$.
   Right: the matrix after column/row permutation.  }
\label{fig:Rpermutation}
\end{figure}

However, there is no need to carry out column/row permutation
explicitly. {\sc Matlab}'s column/row indexing conveniently achieves
this purpose without actual permutation. To be more precise,
$\bA=\bR_{i-1}$ is a $64\times 64$ block, $\bD=\bM_i^\circ({\tt circL}, {\tt
   circL})$ is a $32\times 32$ block, $\bU=\bM_i^\circ({\tt starL}, {\tt
   circL})$ is a $64\times 32$ block, and $\bV=\bM_i^\circ({\tt circL}, {\tt
   starL})$ is a $32\times 64$ block, where {\tt circL}=[1:16\, 81:96]
and {\tt starL}=17:80.  And \cite[Eq.~(31)]{Hels18} follows from the
following block matrix inversion formula \be
\begin{bmatrix} \bA^{-1} & \bU \\ \bV & \bD \end{bmatrix}^{-1}
=\begin{bmatrix} \bA+\bA\bU(\bD-\bV\bA\bU)^{-1}\bV\bA &
-\bA\bU(\bD-\bV\bA\bU)^{-1}\\
-(\bD-\bV\bA\bU)^{-1}\bV\bA & (\bD-\bV\bA\bU)^{-1}\end{bmatrix},
\ee
which can be verified via direct matrix multiplication. Thus, in order
to achieve better efficiency and stability, Algorithm~\ref{alg1}
is replaced by function~{\tt Rcomp} (Figure~\ref{figb2}) that
contains the main loop for the forward recursion, and function~{\tt SchurBana}
(Figure~\ref{figb3}) that replaces steps 8--11 in Algorithm~\ref{alg1}
by \cite[Eq.~(31)]{Hels18}. We note that one only needs to invert a
$32 \times 32$ matrix in function~{\tt SchurBana}.

\begin{figure}[!ht]
\centering
\begin{lstlisting}
   function [z,zp,zpp,nz,w,wzp]=zlocinit(theta,T,W,nsub,i,npan)
   % This function returns to the user the discretization of
   % the type b mesh at the ith level
   %
   % Input parameters:
   % theta  - the parameter for the one-corner curve
   % T, W   - Gauss-Legendre nodes and weights on [-1,1]
   % nsub   - the total number of dyadic refinement in the
   %          calculation of the preconditioner R
   % i      - the level index
   % npan   - the total number of coarse panels
   %
   % Output parameters:
   % z,zp,zpp,nz,w,wzp - complex column vectors of length 96
   % z      - coordinates of discretization points (in complex
   %          number format) from the top panel to the bottom panel
   % zp,zpp - z'(s), z''(s)
   % nz     - unit normal vector
   % w      - quadrature weights
   % wzp    - w*z'(s)

   % 1/npan is the length of the coarse panel in the parameter space
   % h is the length of each panel on the type c mesh at level i
   h=1/npan/2^(nsub-i);

   % s returns discretization points in the parameter space for
   % the bottom 3 panels in the type b mesh at level i.
   % That is, s contains scaled and shifted Gauss-Legendre nodes on
   % the three panels [0, 0.5h], [0.5h, h], [h, 2h].
   s=[T/4+0.25;T/4+0.75;T/2+1.5]*h;
   w=[W/4;W/4;W/2]*h; w=[flipud(w);w];

   z=zfunc(s,theta);
   % conj(flipud(z)) produces points on the top three panels of
   % the type b mesh by symmetry.
   z=[conj(flipud(z));z];

   zp=zpfunc(s,theta);zp=[-conj(flipud(zp));zp];
   zpp=zppfunc(s,theta);zpp=[conj(flipud(zpp));zpp];
   nz=-1i*zp./abs(zp);
   wzp=w.*zp;
\end{lstlisting}
\caption{\sf {\sc Matlab} function {\tt zlocinit}.
}
\label{figb1}
\end{figure}

\begin{figure}[!ht]
\centering
\begin{lstlisting}
   function R=Rcomp(theta,lambda,T,W,Pbc,PWbc,nsub,npan)
   % Forward recursion for computing the nontrivial 64x64 block
   % of the preconditioner R

   % indices for the center four panels on the type b mesh
   starL=17:80;
   % indices for the top and bottom panels on the type b mesh
   circL=[1:16 81:96];

   % forward recursion - level 1 is the finest level,
   % level nsub goes back to four coarse panels near the corner
   for level=1:nsub
     % discretization on the type b mesh at the current level
     [z,zp,zpp,nz,w,wzp]=zlocinit(theta,T,W,nsub,level,npan);
     % construct the interaction matrix on the type b mesh
     K=MAinit(z,zp,zpp,nz,w,wzp,96);
     MAT=eye(96)+lambda*K; % the full system matrix
     % at the finest level, R is initialized by simply inverting
     % the "bad" part of the system matrix
     if level==1
       R=inv(MAT(starL,starL));
     end
     % use block inversion formula to carry out the forward recursion
     R=SchurBana(Pbc,PWbc,MAT,R,starL,circL);
   end
\end{lstlisting}
\caption{\sf {\sc Matlab} function {\tt Rcomp}.
   }
\label{figb2}
\end{figure}

\begin{figure}[!ht]
\centering
\begin{lstlisting}
   function A=SchurBana(P,PW,K,A,starL,circL)
   % This function uses the block matrix inversion formula
   % to compute the forward recursion for the nontrivial
   % 64x64 block of the preconditioner R.
   %
   % Input parameters:
   % P     - 96x64 prolongation matrix
   % PW    - 96x64 weighted prolongation matrix
   % A     - 64x64 preconditioner matrix R_{i-1} at level i-1
   % starL - 17:80 center four panels on the type b mesh
   % circL - [1:16 81:96] top and bottom panels on the type b mesh
   %
   % Output parameters:
   % A     - 64x64 preconditioner matrix R_{i} at level i
   starS=17:48; % center two panels on the type c mesh
   circS=[1:16 49:64]; % top and bottom panels on the type c mesh

   % Using the notation in Eq. (31) in the RCIP tutorial, we have
   % V = K(circL, starL);
   % U = K(starL, circL);
   % D = K(circL, circL);
   % A = R_{i-1};

   VA=K(circL,starL)*A; % V*A
   PTA=PW(starL,starS)'*A; % P_W^T * A
   PTAU=PTA*K(starL,circL); % P_W^T * A * U
   DVAUI=inv(K(circL,circL)-VA*K(starL,circL)); % (D-V A U)^{-1}
   DVAUIVAP=DVAUI*(VA*P(starL,starS)); % (D-V A U)^{-1} * (V A P)

   A(starS,starS)=PTA*P(starL,starS)+PTAU*DVAUIVAP; % (1,1) block
   A(circS,circS)=DVAUI; % (2,2) block
   A(circS,starS)=-DVAUIVAP; % (2,1) block
   A(starS,circS)=-PTAU*DVAUI; % (1,2) block
\end{lstlisting}
\caption{\sf {\sc Matlab} function {\tt SchurBana}.
}
\label{figb3}
\end{figure}

\bibliographystyle{abbrv}
\bibliography{srhs}

\begin{thebibliography}{10}

\bibitem{abram53}
M.~Abramowitz.
\newblock Evaluation of the integral $\int_0^{\infty}e^{-u^2-x/u}du$.
\newblock {\em J. Math. Phys. Camb.}, 32:188--192, 1953.

\bibitem{klinteberg19}
L.~af~Klinteberg, F.~Fryklund, and A.-K. Tornberg.
\newblock An adaptive kernel-split quadrature method for parameter-dependent
  layer potentials.
\newblock {\em arXiv preprint arXiv:1906.07713}, 2019.

\bibitem{ahlfors1966}
L.~V. Ahlfors.
\newblock {\em Complex analysis: an introduction to the theory of analytic
  functions of one complex variable}.
\newblock McGraw-Hill, New York, 1966.

\bibitem{AskhGree14}
T.~Askham and L.~Greengard.
\newblock Norm-preserving discretization of integral equations for elliptic
  {PDE}s with internal layers {I}: the one-dimensional case.
\newblock {\em SIAM Rev.}, 56(4):625--641, 2014.

\bibitem{bgk1954}
P.~L. Bhatnagar, E.~P. Gross, and M.~Krook.
\newblock A model for collision processes in gases. {I}. {S}mall amplitude
  processes in charged and neutral one-component systems.
\newblock {\em Phys. Rev.}, 94(3):511--525, 1954.

\bibitem{Brem12}
J.~Bremer.
\newblock On the {N}ystr{\"o}m discretization of integral equations on planar
  curves with corners.
\newblock {\em Appl. Comput. Harmon. Anal.}, 32(1):45--64, 2012.

\bibitem{Cercignani2000}
C.~Cercignani.
\newblock {\em Rarefied Gas Dynamics: From Basic Concepts to Actual
  Calculations}.
\newblock Cambridge University Press, Cambridge, UK, 2000.

\bibitem{cheng2006cm}
H.~Cheng, W.~Crutchfield, Z.~Gimbutas, L.~Greengard, J.~Huang, V.~Rokhlin,
  N.~Yarvin, and J.~Zhao.
\newblock Remarks on the implementation of the wideband {FMM} for the
  {H}elmholtz equation in two dimensions.
\newblock In {\em Inverse problems, multi-scale analysis and effective medium
  theory}, volume 408 of {\em Contemp. Math.}, pages 99--110. Amer. Math. Soc.,
  Providence, RI, 2006.

\bibitem{colton2012}
D.~Colton and R.~Kress.
\newblock {\em Inverse {A}coustic and {E}lectromagnetic {S}cattering {T}heory}.
\newblock Springer, New York, NY, 2012.

\bibitem{Devi18}
Y.~U. Devi, M.~Rukmini, and B.~Madhav.
\newblock A compact conformal printed dipole antenna for 5{G} based vehicular
  communication applications.
\newblock {\em Prog. Electromagn. Res. C}, 85:191--208, 2018.

\bibitem{ethridge2001sisc}
F.~Ethridge and L.~Greengard.
\newblock A new fast-multipole accelerated {P}oisson solver in two dimensions.
\newblock {\em SIAM J. Sci. Comput.}, 23(3):741--760, 2001.

\bibitem{gimbutas2021siamam}
Z.~Gimbutas.
\newblock Edge and corner preconditioners in three dimensions.
\newblock {SIAM} {A}nnual {M}eeting, 2021.

\bibitem{Gimbutas20}
Z.~Gimbutas, S.~Jiang, and L.-S. Luo.
\newblock Evaluation of {A}bramowitz functions in the right half of the complex
  plane.
\newblock {\em J. Comput. Phys.}, 405:109169, 2020.

\bibitem{greengard1987jcp}
L.~Greengard and V.~Rokhlin.
\newblock A fast algorithm for particle simulations.
\newblock {\em J. Comput. Phys.}, 73(2):325--348, 1987.

\bibitem{Hels09JCP}
J.~Helsing.
\newblock Integral equation methods for elliptic problems with boundary
  conditions of mixed type.
\newblock {\em J. Comput. Phys.}, 228(23):8892--8907, 2009.

\bibitem{Hels18}
J.~Helsing.
\newblock Solving integral equations on piecewise smooth boundaries using the
  {RCIP} method: a tutorial.
\newblock {\em arXiv preprint arXiv:1207.6737v9}, 2018.

\bibitem{HelsJian18}
J.~Helsing and S.~Jiang.
\newblock On integral equation methods for the first {D}irichlet problem of the
  biharmonic and modified biharmonic equations in nonsmooth domains.
\newblock {\em SIAM J. Sci. Comput.}, 40(4):A2609--A2630, 2018.

\bibitem{helsing2014jcp}
J.~Helsing and A.~Karlsson.
\newblock An explicit kernel-split panel-based {N}ystr{\"o}m scheme for
  integral equations on axially symmetric surfaces.
\newblock {\em J. Comput. Phys.}, 272:686--703, 2014.

\bibitem{Hels08}
J.~Helsing and R.~Ojala.
\newblock On the evaluation of layer potentials close to their sources.
\newblock {\em J. Comput. Phys.}, 227(5):2899--2921, 2008.

\bibitem{helsing2013acha}
J.~Helsing and K.-M. Perfekt.
\newblock On the polarizability and capacitance of the cube.
\newblock {\em Appl. Comput. Harmon. Anal.}, 34(3):445--468, 2013.

\bibitem{Hend81}
H.~V. Henderson and S.~R. Searle.
\newblock On deriving the inverse of a sum of matrices.
\newblock {\em SIAM Rev.}, 23(1):53--60, 1981.

\bibitem{High96}
N.~J. Higham.
\newblock {\em Accuracy and stability of numerical algorithms}.
\newblock SIAM, 2002.

\bibitem{hsiao2008}
G.~C. Hsiao and W.~L. Wendland.
\newblock {\em Boundary integral equations}, volume 164 of {\em Applied
  Mathematical Sciences}.
\newblock Springer-Verlag, Berlin, 2008.

\bibitem{Jiang16}
S.~Jiang and L.-S. Luo.
\newblock Analysis and solutions of the integral equation derived from the
  linearized {BGKW} equation for the steady {Couette} flow.
\newblock {\em J. Comput. Phys.}, 316:416--434, 2016.

\bibitem{Kaha65}
W.~Kahan.
\newblock Further remarks on reducing truncation errors.
\newblock {\em Commun. Assoc. Comput. Mach.}, 8:40, 1965.

\bibitem{kress2014}
R.~Kress.
\newblock {\em Linear Integral Equations}, volume~82 of {\em Applied
  Mathematical Sciences}.
\newblock Springer--Verlag, Berlin, third edition, 2014.

\bibitem{Li2015}
W.~Li, L.-S. Luo, and J.~Shen.
\newblock Accurate solution and approximations of the linearized {BGK} equation
  for steady {Couette} flow.
\newblock {\em Comput. Fluids}, 111:18--32, 2015.

\bibitem{Ma96}
J.~Ma, V.~Rokhlin, and S.~Wandzura.
\newblock Generalized {G}aussian quadrature rules for systems of arbitrary
  functions.
\newblock {\em SIAM J. Numer. Anal.}, 33(3):971--996, 1996.

\bibitem{macleod1992anm}
A.~J. MacLeod.
\newblock Chebyshev expansions for {A}bramowitz functions.
\newblock {\em Appl. Numer. Math.}, 10:129--137, 1992.

\bibitem{martinsson2020}
P.-G. Martinsson.
\newblock {\em Fast direct solvers for elliptic {PDE}s}, volume~96 of {\em
  CBMS-NSF Regional Conference Series in Applied Mathematics}.
\newblock Society for Industrial and Applied Mathematics (SIAM), Philadelphia,
  PA, 2020.

\bibitem{nie2001siap}
Q.~Nie and F.-R. Tian.
\newblock Singularities in {H}ele--{S}haw flows driven by a multipole.
\newblock {\em SIAM J. Appl. Math.}, 62(2):385--406, 2001.

\bibitem{pozrikidis1992}
C.~Pozrikidis.
\newblock {\em Boundary integral and singularity methods for linearized viscous
  flow}.
\newblock Cambridge Texts in Applied Mathematics. Cambridge University Press,
  Cambridge, 1992.

\bibitem{rappa17}
T.~S. {Rappaport}, G.~R. {MacCartney}, S.~{Sun}, H.~{Yan}, and S.~{Deng}.
\newblock Small-scale, local area, and transitional millimeter wave propagation
  for 5{G} communications.
\newblock {\em IEEE Trans. Antennas Propag.}, 65(12):6474--6490, 2017.

\bibitem{strang1993}
G.~Strang.
\newblock {\em Introduction to linear algebra}.
\newblock Wellesley-Cambridge Press, Wellesley, MA, 5th edition, 2016.

\bibitem{w1954}
P.~Welander.
\newblock On the temperature jump in a rarefied gas.
\newblock {\em Ark. Fys.}, 7:507--553, 1954.

\end{thebibliography}

\end{document}